\documentclass[11pt, a4paper]{amsart}
\usepackage{amsfonts,amsmath,amssymb, amscd}
\usepackage{txfonts}
\usepackage[all]{xy}

\newtheorem{theorem}{Theorem}[section]
\newtheorem{lemma}[theorem]{Lemma}
\newtheorem{prop}[theorem]{Proposition}
\newtheorem{corollary}[theorem]{Corollary}

\theoremstyle{definition}
\newtheorem{definition}[theorem]{Definition}
\newtheorem{rem}[theorem]{Remark}

\usepackage{color}

\newcommand{\mltimes}{\mathop{\raisebox{0.2ex}{\makebox[0.92em][l]{${\scriptstyle\vartriangleright\mathrel{\mkern-4mu}<}$}}}}
\newcommand{\cmdblackltimes}{\mathop{\raisebox{0.2ex}{\makebox[0.92em][l]{${\scriptstyle\blacktriangleright\mathrel{\mkern-4mu}<}$}}}}

\usepackage{mathrsfs}

\renewcommand{\k}{\Bbbk}
\newcommand{\tens}{\otimes}
\renewcommand{\H}{\mathbf{H}}

\newcommand{\A}{\mathbf{A}}

\newcommand{\T}{\mathbf{T}}
\newcommand{\U}{\mathbf{U}}
\newcommand{\G}{\mathbf{G}}

\newcommand{\hy}{\operatorname{hy}}
\renewcommand{\O}{\mathscr{O}}

\newcommand{\isomto}{\overset{\simeq}{\longrightarrow}}

\numberwithin{equation}{section}

\begin{document} 

\title[Algebraic supergroups]{Algebraic supergroups and Harish-Chandra pairs over a commutative ring}

\author[A.~Masuoka]{Akira Masuoka}
\address{Akira Masuoka: 
Institute of Mathematics, 
University of Tsukuba, 
Ibaraki 305-8571, Japan}
\email{akira@math.tsukuba.ac.jp}

\author[T.~Shibata]{Taiki Shibata}
\address{Taiki Shibata: 
Graduate School of Pure and Applied Sciences, 
University of Tsukuba, 
Ibaraki 305-8571, Japan/Current address: 
Department of Mathematical and Statistical Sciences,
University of Alberta,
Edmonton~AB~T6G~2G1, 
Canada}
\email{tshibata@math.tsukuba.ac.jp/shibata@ualberta.ca}


\begin{abstract}
We prove a category equivalence between algebraic supergroups and Harish-Chandra pairs
over a commutative ring which is $2$-torsion free.
The result is applied to re-construct
the Chevalley $\mathbb{Z}$-supergroups constructed by Fioresi and Gavarini \cite{FGmemo} and by
Gavarini \cite{G21, G1}. For a wide class of algebraic supergroups 
we describe their representations by using their super-hyperalgebras.  
\end{abstract}

\maketitle



\section{Introduction}\label{sec:introduction}

Let $\k$ be a non-zero commutative ring over which we work. 

The word $``$super" is used as a synonym of 
$``$graded by $\mathbb{Z}_2=\{ 0,1 \}$". Ordinary objects, such as Lie/Hopf algebras, which are defined in 
the tensor category of $\k$-modules, given the trivial symmetry $v\otimes w \mapsto w \otimes v$,
are generalized by their super-analogues, such as Lie/Hopf superalgebras, which are defined in the
tensor category of $\mathbb{Z}_2$-graded $\k$-modules, given the super-symmetry \eqref{eq:super-symmetry}.
Our main concern are the super-analogues of affine/algebraic groups. By saying 
\emph{affine groups} (resp., \emph{algebraic groups}), we mean, following 
Jantzen \cite{J}, what 
are formally called affine group schemes (resp., affine algebraic group schemes), and 
we will use analogous simpler names for their super analogues. 
 
An \emph{algebraic supergroup} (over $\k$) is thus a representable group-valued functor $\G$ 
defined on the category of commutative superalgebras over $\k$,
such that the commutative Hopf superalgebra $\O(\G)$ representing $\G$ is finitely generated; see
\cite[Chapter 11]{CCF}, for example. Associated with such $\G$ are a Lie superalgebra,
$\mathrm{Lie}(\G)$, and an algebraic group, $\G_{ev}$. The latter is the (necessarily, representable) 
group-valued functor obtained from $\G$ by restricting the domain to the category of commutative algebras. 

Important examples of algebraic supergroups over the complex number field $\mathbb{C}$ 
are \emph{Chevalley} $\mathbb{C}$-\emph{supergroups}; 
they are the algebraic supergroups $\G$ over $\mathbb{C}$ such that $\mathrm{Lie}(\G)$ 
is one of the complex simple Lie superalgebras, which were classified by Kac \cite{Kac}. 
Just as Kostant \cite{Kostant1} once did in the classical, non-super situation, Fioresi and Gavarini constructed 
natural $\mathbb{Z}$-forms of the Chevalley $\mathbb{C}$-supergroups; see \cite{FGmemo, G21, G1}. 
Those $\mathbb{Z}$-forms, called \emph{Chevalley} $\mathbb{Z}$-\emph{supergroups}, are important,
and would be useful especially to study Chevalley supergroups in positive characteristic. 
A motivation of this paper is to make part of Fioresi and Gavarini's construction 
simpler and more rigorous, and we realize it by using \emph{Harish-Chandra pairs}, as will be explained
below. 
Their construction is parallel to the classical
one; it starts with (1)~proving the existence of $``$Chevalley basis" for each complex simple Lie 
superalgebra $\mathfrak{g}$, and then turns to (2)~constructing from the basis a natural 
$\mathbb{Z}$-form, called a \emph{Kostant superalgebra}, of $\mathbf{U}(\mathfrak{g})$.  
Our construction, which will be given in Section \ref{sec:Chevalley}, 
uses results from these (1) and (2), but dispenses with the following procedures, which include
to choose a faithful representation of $\mathfrak{g}$ on a finite-dimensional complex super-vector space
including an appropriate $\mathbb{Z}$-lattice; see Remarks \ref{rem:FGconstruction} and \ref{rem:Gconstruction}.

In this and the following paragraphs, 
let us suppose that $\k$ is a field of characteristic $\ne 2$. Even in this case, algebraic supergroups
have not been studied so long as Lie supergroups. Indeed, the latter has a longer history of study founded
by Kostant \cite{Kostant2}, Koszul \cite{Koszul} and others in the 1970's. An important result from the study
is the equivalence, shown by Kostant, between the category of Lie supergroups and the category of
Harish-Chandra pairs; see \cite[Section 7.4]{CCF}, \cite{V}.  
The corresponding result for algebraic supergroups, that is, the equivalence
\begin{equation}\label{eq:equivalence}
\mathsf{ASG}\approx \mathsf{HCP}
\end{equation}
between the category $\mathsf{ASG}$ of algebraic supergroups and the category $\mathsf{HCP}$ of Harish-Chandra
pairs, was only recently proved by Carmeli and Fioresi \cite{CF} when $\k = \mathbb{C}$, and then by the 
first-named author \cite{M2} for an arbitrary field of characteristic $\ne 2$; see \cite{M2, GZ} for applications of the result.  
As was done for Lie supergroups, Carmeli and Fioresi define a \emph{Harish-Chandra pair} to be a pair 
$(G, \mathfrak{g})$
of an algebraic group $G$ and a finite-dimensional Lie superalgebra $\mathfrak{g}$ which satisfy 
some conditions (see Definition \ref{def:HCP}), and proved that the equivalence \eqref{eq:equivalence}
is given by $\G \mapsto (\G_{ev}, \mathrm{Lie}(\G))$ (see the third paragraph above). 
In \cite{M2}, the definition of 
Harish-Chandra pairs and the category equivalence are given by purely Hopf algebraic terms, but
they will be easily seen to be essentially the same as those in \cite{CF} and in this paper;
see Remarks \ref{rem:compare_definitions}~(1) and \ref{rem:compare_equivalences}. 

To prove the category equivalence, the articles \cite{CF} and \cite{M2} both use the following 
property of $\O(\G)$, which was proved in \cite{M1} and will be re-produced
as Theorem \ref{thm:tensor_product_decomposition} below: given $\G \in \mathsf{ASG}$, the
Hopf superalgebra $\O(\G)$ is \emph{split} in the sense that there exists 
a counit-preserving isomorphism 
\begin{equation}\label{eq:splitting_property}
\O(\G) \simeq \O(\G_{ev}) \otimes \wedge(W)
\end{equation}
of left $\O(\G_{ev})$-comodule superalgebras, where $W$ is the
odd component of the cotangent super-vector space of $\G$ at $1$, and 
$\wedge(W)$ is the exterior algebra on it. This basic property played a role in \cite{MZ} as well; see
also \cite{M3}. As another application of the property we will prove a representation-theoretic result, 
Corollary \ref{cor:category_isom}, which generalizes results which were proved 
in \cite{BrKl, BrKuj, ShuWang} for some special algebraic supergroups.  

Throughout in the text of this paper we assume that $\k$ is a non-zero commutative ring 
which is $2$-\emph{torsion free}, or namely, is such that 
an element $a \in \Bbbk$ must be zero whenever $2a=0$. We chose this assumption 
because it seems natural, in order to keep the super-symmetry \eqref{eq:super-symmetry} non-trivial. 
Our main result, Theorem \ref{thm:equivalence}, proves the category equivalence \eqref{eq:equivalence}
over such $\k$ as above. We pose some assumptions to objects in the relevant categories,
which are necessarily satisfied if $\k$ is a field. 
Indeed, an algebraic supergroup $\G$ in $\mathsf{ASG}$
is required to satisfy, in particular, the condition that $\O(\G)$ is split, while an object $(G,\mathfrak{g})$
in $\mathsf{HCP}$ is required to satisfy, in particular, the condition that $\mathfrak{g}$ is \emph{admissible}
(see Definition \ref{def:admissible_Lie}), and so, 
given an odd element $v \in \mathfrak{g}_1$, 
the even component $\mathfrak{g}_0$ of $\mathfrak{g}$ must contain a unique element, 
$\frac{1}{2}[v,v]$, whose double equals $[v,v]$; see Section \ref{subsec:P} and Definition \ref{def:HCP} for 
the precise definitions of $\mathsf{ASG}$ and $\mathsf{HCP}$, respectively.   
A novelty of our proof of the result is to construct a functor 
$\G : \mathsf{HCP} \to \mathsf{ASG}$, which will be proved an equivalence, as follows; 
given $(G, \mathfrak{g}) \in \mathsf{HCP}$, we realize the Hopf superalgebra $\O(\G)$
corresponding to $\G = \G(G, \mathfrak{g})$ 
as a discrete Hopf super-subalgebra of some
complete topological Hopf superalgebra, $\widehat{\mathscr{A}}$, that is simply constructed from the given pair.
Indeed, this Hopf algebraic idea was used in \cite{M2}, but our construction has been modified as to be applicable
when $\k$ is a commutative ring.
Based on the proved equivalence 
we will re-construct the Chevalley $\mathbb{Z}$-supergroups, by giving the corresponding Harish-Chandra pairs.

The category equivalence theorem, Theorem \ref{thm:equivalence}, is proved in 
Section \ref{sec:category-equivalence_theorem}, while the Chevalley $\mathbb{Z}$-supergroups
are re-constructed in Section \ref{sec:Chevalley}.
The contents of the remaining three sections are as follows. Section \ref{sec:preliminaries} is
devoted to preliminaries on Hopf superalgebras and affine/algebraic supergroups. In Section \ref{sec:Admissible_LSA},
admissible Lie superalgebras are discussed. Especially, we prove in Corollary \ref{cor:co-split}
that the universal envelope $\U(\mathfrak{g})$
of such a Lie superalgebra $\mathfrak{g}$ has the property which is dual to the splitting property \eqref{eq:splitting_property};
the corollary plays a role in the proof of our main result. In Section \ref{sec:G-supermodules}, we discuss supermodules
over an algebraic supergroup $\G$ and over the super-hyperalgebra $\hy(\G)$ of $\G$, when $\G_{ev}$
is a split reductive algebraic group. Let $T$ be a split
maximal torus of such $\G_{ev}$. Theorem \ref{thm:bijection_super_case}
shows, roughly speaking, equivalence of $\G$-supermodules with $\hy(\G)$-$T$-supermodules. 
When $\k$ is a field, the theorem gives Corollary \ref{cor:category_isom} cited before.

After an earlier version of this paper was submitted, the article \cite{G2} by Gavarini was in circulation.
Theorem 4.3.14 of \cite{G2} essentially proves our category equivalence theorem in the generalized situation that 
$\Bbbk$ is an arbitrary commutative ring. A point is to use the additional structure, called $2$-\emph{operations},
on Lie superalgebras $\mathfrak{g}$, which generalizes the map 
$v\mapsto \frac{1}{2}[v,v]$, $\mathfrak{g}_1 \to \mathfrak{g}_0$ given on an admissible Lie superalgebra in our
situation. 
Given a Harish-Chandra pair, Gavarini constructs an affine supergroup in a quite different
method from ours, realizing it as a group sheaf in the Zariski topology. 
In the appendix of this paper we will refine his category
equivalence, using our construction and giving detailed arguments on $2$-operations, in particular.
This would not be meaningless because such detailed arguments are not be given in \cite{G2}; see 
Remark \ref{rema:compare_functors}.

\section{Preliminaries}\label{sec:preliminaries}

\subsection{}\label{subsec:base_ring}
We work over a non-zero commutative ring $\Bbbk$. 
Throughout in what follows except in the appendix,
we assume that $\Bbbk$ is $2$-\emph{torsion free};
this means that an element $a \in \Bbbk$ must be zero whenever $2a=0$. 
It follows that any flat $\Bbbk$-module
is 2-torsion free.

A $\k$-module is said to be $\k$-\emph{finite} (resp., $\k$-\emph{finite free}/\emph{projective}) 
if it is finitely generated (resp., finitely generated and free/projective). 

The unadorned $\otimes$ denotes the tensor product over $\k$. We let $\mathrm{Hom}$ denote the $\k$-module
consisting of $\k$-linear maps. Given a $\k$-module $V$, we let $V^*$ denote the dual $\k$-module
$\mathrm{Hom}(V,\k)$ of $V$. 

\subsection{}\label{subsec:super-objects}
A \emph{supermodule} (over $\k$) is precisely a $\k$-module $V = V_0 \oplus V_1$ graded by the group 
$\mathbb{Z}_2=\{ 0, 1 \}$ of order $2$. The degree of a homogeneous element $v \in V$ is denoted by $|v|$.
Such an element is said to be \emph{even} (resp., \emph{odd}) if $|v|=0$ (resp., if $|v|=1$).  
We say that $V$ is \emph{purely even} (resp., \emph{purely odd}) if $V = V_0$ (resp., if $V=V_1)$. 
The supermodules $V$, $W, \dots$ and the $\mathbb{Z}_2$-graded (or super-)linear maps
naturally form a tensor category $\mathsf{SMod}$; the 
tensor product is the $\k$-module $V \otimes W$ graded so that 
$(V \otimes W)_i = \bigoplus_{j+k=i} V_j\oplus W_k$, $i=0,1$, 
and the unit object is $\k$, which is supposed to be purely even. The tensor category is symmetric 
with respect to the so-called \emph{super-symmetry} 
$c_{V,W} : V \otimes W \overset{\simeq}{\longrightarrow} W \otimes V$ defined by 
\begin{equation}\label{eq:super-symmetry}
c_{V,W}(v\tens w) = (-1)^{|v||w|}w \otimes v = \begin{cases}
-w\tens v & \text{if $v$, $w$ are odd,}\\
\phantom{-}w\tens v & \text{otherwise.}
\end{cases}
\end{equation}

The dual $\k$-module $V^*$ of a supermodule $V$ is a supermodule graded so that 
$(V^*)_i = (V_i)^*$, $i=0,1$. 

Ordinary objects, such as Lie algebras or Hopf algebras, defined in the symmetric 
category of $\k$-modules are generalized by super-objects,
such as Lie superalgebras or Hopf superalgebras, defined in $\mathsf{SMod}$. The ordinary objects are
regarded as purely even super-objects. 

A superalgebra (resp., super-coalgebra) is said to be \emph{commutative} (resp., \emph{cocommutative}) 
if the product (resp., coproduct) is invariant, composed with the super-symmetry. 

\subsection{}\label{subsec:Hopf_pairing}
Given a Hopf superalgebra $\A$, we denote the coproduct, the counit and the antipode by
\[ 
\Delta : \A \to \A \otimes \A,\ \Delta(a) = a_{(1)}\otimes a_{(2)},\quad \varepsilon : \A \to \k,\quad
S : \A \to \A, 
\]
respectively. The antipode $S$ preserves the unit and the counit, and satisfies
\[ 
m \circ (S \otimes S) = S \circ m \circ c_{\A,\A},
\quad (S \otimes S)\circ \Delta=c_{\A,\A}\circ \Delta \circ S,
\]
where $m : \A \otimes \A \to \A$ denotes the product. 
We let $\A^+$ denote the augmentation super-ideal $\operatorname{Ker}\varepsilon$ of $\A$. 

Let $\H$, $\A$ be Hopf superalgebras. A bilinear map $\langle \ , \ \rangle : \H \times \A \to \k $ 
is called a \emph{Hopf pairing} \cite[Section 2.2]{M2}, 
if 
$\langle \H_i,\A_j \rangle = 0$ whenever $i \ne j$,
and if we have
\begin{align}\label{eq:Hopf_pairing_conditions}
 &\langle xy, a \rangle = \langle x, a_{(1)}\rangle\, \langle y, a_{(2)}\rangle, \notag \\
 &\langle x, ab \rangle = \langle x_{(1)}, a\rangle\, \langle x_{(2)}, b\rangle, \\
 &\langle 1, a \rangle = \varepsilon (a), \quad\langle x, 1 \rangle = \varepsilon (x), \notag
\end{align}
where $x, y \in \H$, $a, b \in \A$. The last conditions imply
\begin{equation}\label{eq:Hopf_pairing_antipode}
\langle S(x), a \rangle  = \langle x, S(a) \rangle, \quad x\in \H,\ a \in \A. 
\end{equation}

Let $V$ be a $\k$-module, and regard it as a purely odd supermodule. The tensor algebra $\mathbf{T}(V)$ 
on $V$ uniquely turns into a Hopf superalgebra in which every element of $V$ is an odd primitive. 
The \emph{exterior algebra} $\wedge(V)$ on $V$ is the quotient 
Hopf superalgebra of $\mathbf{T}(V)$ by the Hopf super-ideal generated by the even primitives $v^2$, where
$v \in V$. Note that $\mathbf{T}(V)$ is cocommutative, while $\wedge(V)$ is commutative and cocommutative.
Suppose that $V$ is $\k$-finite free. Then $\wedge(V)$ is $\k$-finite free, so that the dual supermodule
$\wedge(V)^*$, given the ordinary, dual-algebra and dual-coalgebra structures, is a Hopf superalgebra;
see \cite[Remark~1]{M2}. The canonical pairing 
$\langle \hspace{2mm}, \ \rangle : V \times V^* \to \Bbbk$ 
uniquely extends to a Hopf pairing 
$\langle \hspace{2mm}, \ \rangle : \wedge(V) \times \wedge(V^*) \to \Bbbk$;  
it is determined by the property that $\langle \wedge^m(V), \wedge^n(V^*) \rangle = 0$ unless $m=n$, and
by the formula
\begin{equation}\label{cano-pairing}
\langle v_1 \wedge \dots \wedge v_n, \, w_1 \wedge \dots \wedge w_n \rangle = 
\sum_{ \sigma \in \mathfrak{S}_n }\operatorname{sgn} \sigma \, \langle v_1, w_{ \sigma (1) } 
\rangle \dots \langle v_n, w_{ \sigma (n) } \rangle, 
\end{equation}
where $v_i \in V,\ w_i \in V^*,\ n > 0$. Since this Hopf pairing is non-degenerate, it induces the
isomorphism $\wedge(V^*)\overset{\simeq}{\longrightarrow} \wedge(V)^*$, $a \mapsto \langle \ , a\rangle$\
of Hopf superalgebras, through which we will identify as 
\[
\wedge(V^*)=\wedge(V)^*.
\]

\subsection{}\label{subsec:HSA_basics}
Let $\A$ be a commutative Hopf superalgebra. Define 
\[
\overline{\A} := \A/(\A_1),\quad W^{\A}:= \A_1/\A_0^+\A_1, 
\]
where $(\A_1)$ denotes the (Hopf super-)ideal of $\A$ generated by the odd component $\A_1$,
and $\A_0^+ = \A_0 \cap \A^+$; see \cite[Section 4]{M1}. Note that $\overline{\A}=\A_0/\A_1^2$, and
this is the largest purely even quotient Hopf superalgebra of $\A$. We denote the quotient map by
\begin{equation}\label{eq:quotient_map}
\A \to \overline{\A},\quad a \mapsto \overline{a}.  
\end{equation}
We regard $\A$ as a left $\overline{\A}$-comodule superalgebra, naturally, by 
$\A \to \overline{\A}\otimes \A$,
$a \mapsto \overline{a}_{(1)}\otimes a _{(2)}$. Similarly, $\A$ is regarded as a right $\overline{\A}$-comodule superalgebra.

\begin{definition}\label{def:split}
$\A$ is said to be \emph{split}, if $W^{\A}$ is $\k$-free, and if
there exists an isomorphism 
$\psi : \A \overset{\simeq}{\longrightarrow}
\overline{\A}\otimes \wedge(W^{\A})$ of left $\overline{\A}$-comodule superalgebras. 
\end{definition}
\begin{rem}\label{rem:split}
(1)\
If the second condition above is satisfied, then $\psi$ can be re-chosen as \emph{counit-preserving} 
in the sense that $ (\varepsilon \otimes \varepsilon) \circ \psi = \varepsilon $.  Indeed,
if we set $\gamma := (\varepsilon \otimes \varepsilon) \circ \psi$, then 
$a \mapsto \psi(a_{(1)})\, \gamma\circ S(a_{(2)})$ is seen to be a counit-preserving isomorphism. 

(2)\ 
The same condition as above is equivalent to the condition with the sides switched, that is, the condition
that there exists a (counit-preserving) isomorphism 
$\A \overset{\simeq}{\longrightarrow}
\wedge(W^{\A})\otimes \overline{\A}$ of right $\overline{\A}$-comodule superalgebras. Indeed, if $\psi$
is a left or right-sided isomorphism, then the composite 
$c \circ \psi \circ S$, where $c = c_{\overline{\A},\wedge(W^{\A})}$ or
$c_{\wedge(W^{\A}),\overline{\A}}$, 
gives an opposite-sided one.
\end{rem}

\begin{theorem}[\text{\cite[Theorem 4.5]{M1}}]\label{thm:tensor_product_decomposition}
If $\k$ is a field of characteristic $\ne 2$, then every commutative Hopf superalgebra is split.
\end{theorem}

A Hopf superalgebra is said to be \emph{affine} if it is commutative and finitely generated.
A split commutative Hopf superalgebra $\A$ is affine if and only if $\overline{\A}$ is affine
and $W^{\A}$ is $\k$-finite (free). 

All commutative Hopf superalgebras and all Hopf superalgebra maps form a category. The affine
Hopf superalgebras form a full subcategory of the category.

\subsection{}\label{subsec:ASG_basics}
The notions of \emph{affine groups} and of \emph{algebraic groups} (see \cite[Part~I, 2.1]{J}) are 
directly generalized to the super-situation, as follows. 
A \emph{supergroup} is a functor from the category of commutative superalgebras to the category of groups. 
An \emph{affine supergroup} $\G$ is a representable supergroup. 
By Yoneda's Lemma it is represented by a uniquely determined, commutative
Hopf superalgebra, which we denote by $\O(\G)$. We call $\G$ an \emph{algebraic supergroup} if $\O(\G)$ is affine.

The category formed by all 
affine supergroups and all natural transformations of group-valued functors is anti-isomorphic
to the category of commutative Hopf superalgebras.  The full subcategory of the former category
which consists of all algebraic supergroups is anti-isomorphic to the category of affine Hopf superalgebras. 

Let $\G$ be an affine supergroup, and set $\A:=\O(\G)$. Then $\overline{\A}$ represents the supergroup
\[ R \mapsto \G(R_0), \]
where $R$ is a commutative superalgebra. This affine supergroup is denoted by $\G_{ev}$, so that
\[ \overline{\A} = \O(\G_{ev}). \]
We will often regard $\G_{ev}$ as the affine group corresponding to the commutative Hopf algebra
$\overline{\A}$. 
One sees that 
$W^{\A}$ is the odd component of the cotangent supermodule $\A^+/(\A^+)^2$ of $\G$ at $1$.

\subsection{}\label{subsec:representation}
Let $\G$ be an affine supergroup.

Given a supermodule $W$, the left (resp., right) $\G$-supermodule structures on $W$ correspond precisely to 
the right (resp., left) $\O(\G)$-super-comodule structures on $W$. 

Let 
$W \to W \otimes \O(\G)$, $w \mapsto w^{(0)}\otimes w^{(1)}$
be a right $\O(\G)$-super-comodule structure . 
The corresponding left $\G$-supermodule structure
is given by the $R$-super-linear automorphism of $W\otimes R$ which is defined by
\[
{}^{\gamma}(w \otimes 1) = w^{(0)}\otimes \gamma(w^{(1)}), \quad \gamma \in \G(R),\ w \in W,
\]
where $R$ is an arbitrary commutative superalgebra. 
For simplicity this left (resp., the analogous right) 
$\G$-supermodule structure is represented as
\begin{equation}\label{eq:convention}
{}^{\gamma}w\quad (\text{resp.,}\ w^{\gamma}),\quad \gamma\in \G,\ w \in W.
\end{equation}
Actually, this notational convention will be applied 
only when $\G$ is an affine group. 

Given a Hopf pairing $\langle \ , \ \rangle : \H \times \O(\G)\to \k$, where $\H$ is a
Hopf superalgebra, there is induced the left $\H$-supermodule structure on $W$ defined by 
\begin{equation}\label{eq:induced_supermodule}
x w := w^{(0)}\, \langle x, w^{(1)}\rangle,\quad x \in \H,\ w \in W. 
\end{equation}
Similarly, a right $\H$-supermodule structure is induced from a right $\G$-supermodule structure. 

Let $G$ be an affine group. Note that a $G$-supermodule is a supermodule $W$ given a
$G$-module structure such that each component $W_i$, $i=0,1$, is $G$-stable. We let 
\begin{equation}\label{eq:G-SMod} 
G\text{-}\mathsf{SMod}\quad (resp., \mathsf{SMod}\text{-}G)
\end{equation}
denote the category of left (resp., right) $G$-supermodules. This is naturally a tensor category, and is 
symmetric with respect to the super-symmetry.

\section{Admissible Lie superalgebras}\label{sec:Admissible_LSA}
\subsection{}\label{subsec:definition_of_admissible_LSA}
A \emph{Lie superalgebra} is a supermodule $\mathfrak{g}$, given a super-linear map
$[\ , \ ] : \mathfrak{g} \otimes \mathfrak{g} \to \mathfrak{g}$, called a \emph{super-bracket}, which satisfies 
\begin{itemize}
\item[(i)] $[u,u]=0,\ u \in \mathfrak{g}_0$,
\item[(ii)] $[[v,v],v] =0,\ v \in \mathfrak{g}_1$,
\item[(iii)]
$[\ , \ ] \circ (\mathrm{id}_{\mathfrak{g} \otimes \mathfrak{g}} + c_{\mathfrak{g},\mathfrak{g}})= 0$,\ and
\item[(iv)] 
$[[\ , \ ],\ ] \circ (\mathrm{id}_{\mathfrak{g}\otimes \mathfrak{g}\otimes \mathfrak{g}} + 
c_{\mathfrak{g}, \mathfrak{g}\otimes \mathfrak{g}} + c_{\mathfrak{g} \otimes \mathfrak{g}, \mathfrak{g}}) = 0$.
\end{itemize}
Note that (i) implies the equation (iii) restricted to $\mathfrak{g}_0^{\otimes 2}$. If 
$\mathfrak{g}_1$ is $2$-torsion free, then (ii) and (iii) imply the equation (iv) restricted to
$\mathfrak{g}_1^{\otimes 3}$. Indeed, this follows by applying (ii) to the sum $v_1+v_2+v_3$ of elements
$v_i \in \mathfrak{g}_1$. 

A \emph{Lie algebra} is a $\k$-module with a bracket which satisfies (i) and the Jacobi identity, that is,
(iv) in the purely even situation; it is, therefore, the same as a purely even Lie superalgebra. It follows
that if $\mathfrak{g}$ is a Lie superalgebra, then $\mathfrak{g}_0$ is a Lie algebra.

\begin{definition}\label{def:admissible_Lie}
A Lie superalgebra $\mathfrak{g}$ is said to be \emph{admissible} if
\begin{itemize}
\item[(A1)] $\mathfrak{g}_0$ is $\k$-flat,
\item[(A2)] $\mathfrak{g}_1$ is $\k$-free and
\item[(A3)] for every $v \in \mathfrak{g}_1$, the element $[v,v]$ in $\mathfrak{g}_0$ is $2$-\emph{divisible};
this means that there exists an element $u \in \mathfrak{g}_0$ such that $[v,v]=2u$. 
\end{itemize}
Note that (A3) is satisfied if $\mathfrak{g}_1$ has a $\k$-free basis $X$ such that for every $x \in X$,
$[x,x]$ is $2$-divisible in $\mathfrak{g}_0$. 
\end{definition}
\begin{rem}\label{rem:admissible_Lie}
For any $2$-divisible element $w$ in a $2$-torsion free $\k$-module, the element $u$ such that
$w=2u$ is unique, and it will be denoted by $\frac{1}{2}w$. By (A1) above, this can apply to 
the even component of any admissible Lie superalgebra, so that we have 
$\frac{1}{2}[v,v]$ by (A3).  
\end{rem}

\subsection{}\label{subsec:co-splitting_result}
Let $\mathfrak{g}$ be an admissible Lie superalgebra. The tensor algebra $\T(\mathfrak{g})$ on
$\mathfrak{g}$ uniquely turns into a cocommutative Hopf superalgebra in which every even (resp., odd) element
of $\mathfrak{g}$ is an even (resp., odd) primitive.
The \emph{universal envelope} $\U(\mathfrak{g})$ 
of $\mathfrak{g}$ is the quotient Hopf superalgebra of $\T(\mathfrak{g})$ by the Hopf 
super-ideal generated by the homogeneous primitives
\begin{equation}\label{eq:homog_primitives}
zw -(-1)^{|z||w|}wz-[z,w],\quad v^2 - \frac{1}{2}[v, v], 
\end{equation}
where $z$ and $w$ are homogeneous elements in $\mathfrak{g}$, and $v \in \mathfrak{g}_1$. 
We remark that if $2$ is invertible in $\k$, then the second elements $v^2 - \frac{1}{2}[v, v]$ in
\eqref{eq:homog_primitives} may be removed since they are covered by the first. 
The universal envelope
$U(\mathfrak{g}_0)$ of the Lie algebra $\mathfrak{g}_0$ 
is thus defined, as usual, to be the quotient algebra of 
the tensor algebra $T(\mathfrak{g}_0)$ by the ideal generated by
$zw-wz-[z,w]$, where $z,w\in \mathfrak{g}_0$; this is a cocommutative Hopf algebra.  
The $\k$-flatness assumption (A1) on $\mathfrak{g}_0$ ensures the following.

\begin{lemma}[see \cite{H}]\label{lem:flat_implies_injection}
The canonical map 
$\mathfrak{g}_0\to U(\mathfrak{g}_0)$ is an injection.
\end{lemma}

Through the injection above we will suppose $\mathfrak{g}_0\subset U(\mathfrak{g}_0)$. 
The inclusion $\mathfrak{g}_0 \subset \mathfrak{g}$ induces
a Hopf superalgebra map $U(\mathfrak{g}_0)\to \U(\mathfrak{g})$, by which we will regard $\U(\mathfrak{g})$
as a $U(\mathfrak{g}_0)$-ring, and in particular as a left and right $U(\mathfrak{g}_0)$-module. Recall
that given an algebra $R$, an $R$-\emph{ring} \cite[p.195]{B} is an algebra given an algebra map from $R$. 

\begin{prop}\label{prop:co-split}
$\U(\mathfrak{g})$ is free as a left as well as right $U(\mathfrak{g}_0)$-module. 
In fact, if $X$ is an arbitrary $\k$-free basis of $\mathfrak{g}_1$ given a total order, then the products 
\[ 
x_1\dots x_n,\quad x_i \in X,\ x_1<\dots < x_n,\ n\ge 0 
\]
in $\U(\mathfrak{g})$ form a $U(\mathfrak{g}_0)$-free basis, where $x_i$ in the product
denotes the image of the element under the canonical map $\mathfrak{g}\to \U(\mathfrak{g})$.   
\end{prop}

This is proved in \cite[Lemma~11]{M2}, in the generalized situation treating dual Harish-Chandra pairs, 
but over a field of characteristic $\ne 2$.
Our proof of the proposition will confirm the proof of the cited lemma in our present 
situation. To use the same notation as in \cite{M2} we set 
\begin{equation*}\label{eq:JV} 
J := U(\mathfrak{g}_0), \quad V := \mathfrak{g}_1.  
\end{equation*}
Then the right adjoint action 
\begin{equation}\label{eq:right_adjoint}
\mathrm{ad}_r(u)(v) = [v, u],\quad u \in \mathfrak{g}_0,\ v \in V
\end{equation}
by $\mathfrak{g}_0$ on $V$ uniquely gives rise to a right $J$-module structure on $V$,
which we denote by $v \triangleleft a$, where $v \in V$, $a \in J$. 
If $i : V \to \U(\mathfrak{g})$ denotes the canonical map, we have
\begin{equation}\label{eq:J-action} 
i(v \triangleleft a)= S(a_{(1)})\, i(v)\, a_{(2)},\quad v \in V,\ a \in J
\end{equation}
in $\U(\mathfrak{g})$. Indeed, this follows by induction on the largest length $r$, when we
express $a$ as a sum of elements $u_1\dots u_r$, where $u_i \in \mathfrak{g}_0$.  

\begin{lemma}\label{lem:DHCP}
The right $J$-module structure on $V$ and the super-bracket 
$[ \ , \ ] : V\otimes V \to \mathfrak{g}_0 \subset J$ restricted to $V$ make $(J, V)$ into
a dual Harish-Chandra pair \cite[Definition~6]{M2}, or explicitly we have
\begin{itemize}
\item[(a)] $[u \triangleleft a_{(1)}, v \triangleleft a_{(2)}] = S(a_{(1)})[u, v]a_{(2)},$
\item[(b)] $[u, v] = [v, u]$ and
\item[(c)] $v \triangleleft [v, v] = 0$
\end{itemize}
for all $u, v \in V$, $a \in J$. Properties (b), (c) implies
\begin{itemize}
\item[(d)] $u \triangleleft [v, w] + v \triangleleft [w, u] + w \triangleleft [u, v] = 0,\quad
u,v,w\in V. $
\end{itemize}
\end{lemma}

We remark that (a) is an equation in $\mathfrak{g}_0$, and 
the product of the right-hand side is computed in $J$, which is
possible since $\mathfrak{g}_0 \subset J$.  

\begin{proof}[Proof of Lemma \ref{lem:DHCP}]
One verifies (a), just as proving \eqref{eq:J-action}. Properties (b), (c) are those of Lie superalgebras.  
One sees that (b), applied to $u+v+w$ and combined with (c), implies (d).
\end{proof}
   
\begin{proof}[Proof of Proposition \ref{prop:co-split}]
We will prove only the left $J$-freeness. The result with the antipode applied shows the right $J$-freeness. 

Let $X$ be a totally ordered basis of $V$. We confirm the proof of \cite[Lemma~11]{M2} as follows.
First, we introduce the same order as in the proof into all words in the letters from $X \cup \{ * \}$,
where $*$ stands for any element of $J$. Second, we see by using \eqref{eq:J-action}
that the $J$-ring $\U(\mathfrak{g})$ 
is generated by $X$, and is defined by the reduction system consisting of
\begin{itemize}
\item[(i)] $x a \to a_{(1)}(x \triangleleft a_{(2)}),\quad x \in X,\ a \in J,$ 
\item[(ii)] $xy \to -yx + [x, y],\quad x, y \in X,\ x > y,$
\item[(iii)] $x^2 \to \frac{1}{2} [x, x],\quad x \in X, $
\end{itemize}
where we suppose that in (i), $x \triangleleft a_{(2)}$ is presented as a $\Bbbk$-linear 
combination of elements in $X$. 
Third, we see that the reduction system satisfies the assumptions required by
Bergman's Diamond Lemma \cite[Proposition 7.1]{B}, indeed its opposite-sided version. 

To prove the desired result from the Diamond Lemma, it remains to verify the following
by using the properties (a)--(d) in Lemma \ref{lem:DHCP}:
the overlap ambiguities which may occur when we reduce the words
\begin{itemize}
\item[(iv)]  $xya,\quad x \ge y\ \text{in}\ X,\ a \in J$, 
\item[(v)]   $xyz,\quad x \ge y \ge z\ \text{in}\ X$
\end{itemize}
are all resolvable. The proof of \cite[Lemma~11]{M2} verifies the resolvability 
only when $x, y$ and $z$ are distinct, and the same proof works now as well.  

As for the remaining cases (omitted in the cited proof), first let $xya$ be a word from (iv) with $x = y$.  
This is reduced on the one hand as
\[ 
x x a  \to x a_{(1)}(x \triangleleft a_{(2)})
\to a_{(1)}(x \triangleleft a_{(2)})(x \triangleleft a_{(3)}), 
\]
and on the other hand as
\begin{eqnarray*}
x x a \to \Big(\frac{1}{2}[x, x]\Big)a &=& a_{(1)}S(a_{(2)})\Big(\frac{1}{2}[x,x]\Big)a_{(3)}\\
&=& a_{(1)}\Big(\frac{1}{2}[x \triangleleft a_{(2)}, x\triangleleft a_{(3)}]\Big). 
\end{eqnarray*}
Let $b \in J$. The last equality holds since 
$S(b_{(1)})(\frac{1}{2}[x,x])b_{(2)}$ and 
$\frac{1}{2}[x \triangleleft b_{(1)}, x\triangleleft b_{(2)}]$ 
coincide since their doubles do by (a). 
For the desired resolvability it suffices to see that the two polynomials
\begin{equation}\label{two_polynomials}
(x \triangleleft b_{(1)}) (x \triangleleft b_{(2)}),\quad 
\frac{1}{2}[x \triangleleft b_{(1)}, x\triangleleft b_{(2)}]
\end{equation}
are reduced to the same one. For this, suppose 
\[ (x \triangleleft b_{(1)})\otimes (x \triangleleft b_{(2)}) = \sum_{i, j = 1}^n t_{ij} \, x_i \otimes x_j\ \text{in}\ 
V \otimes V, \]
where $t_{ij} \in \Bbbk$, and $x_1 < \dots < x_n$ in $X$.  Note that $t_{ij} = t_{ji}$ since $J$ is cocommutative. 
Then the first polynomial in \eqref{two_polynomials} is reduced as
\[
\sum_{i < j}t_{i j}(x_ix_j + x_j x_i)+ \sum_i t_{i i}\, x_ix_i \to \sum_{i<j}t_{i j}[x_i, x_j] +\sum_i t_{i i} \Big(\frac{1}{2}[x_i, x_i]\Big).
\]
This last and the second polynomial in \eqref{two_polynomials} coincide since by (b), their doubles do. 
This proves the desired result.  

Next, let $xyz$ be a word from (v), and suppose $x = y > z$. Note that if
$(u,w)=([x,z],x)$ or $(\frac{1}{2} [x, x],z)$, then $u$ is primitive, and so we have the
reduction $wu \to uw + w \triangleleft u$ given by (i). 
Then it follows that $xyz = xxz$ is reduced as
\begin{eqnarray*}
x x z &\to & -x z x + x[x,z] \to z x x - [x,z]x + [x,z]x + x \triangleleft [x,z] \\
      &\to & z\Big(\frac{1}{2}[x,x]\Big) + x \triangleleft [x,z] 
\to \Big(\frac{1}{2}[x,x]\Big)z + z \triangleleft \Big(\frac{1}{2}[x,x]\Big) + x \triangleleft [x,z].  
\end{eqnarray*}
The word is alternatively reduced as
\[
x x z \to \Big( \frac{1}{2}[x,x] \Big) z. 
\] 
These two results coincide, 
since the element $z \triangleleft (\frac{1}{2}[x,x]) + x \triangleleft [x,z]$,
whose double is zero by (d), is zero. The ambiguity for the word $xyz$ is thus
resolvable when $x=y>z$. One proves similarly the resolvability in the remaining
cases, $x > y = z$ and $x = y = z$, using (d) and (c), respectively.     
\end{proof} 

The proposition just proven shows the following. 

\begin{corollary}\label{cor:co-split}
If $\mathfrak{g}$ is an admissible Lie superalgebra, then there exists a unit-preserving,
left $U(\mathfrak{g}_0)$-module super-coalgebra isomorphism
\[
U(\mathfrak{g}_0)\otimes \wedge(\mathfrak{g}_1) \overset{\simeq}{\longrightarrow}
\U(\mathfrak{g}).
\]
Here, $``$unit-preserving" means that the isomorphism sends $1\otimes 1$ to $1$. 
\end{corollary}

\section{Algebraic supergroups and Harish-Chandra pairs}\label{sec:category-equivalence_theorem}

\subsection{}\label{subsec:Lie(G)}
Let $\G$ be an affine supergroup. Set $\A := \O(\G)$. Then the following is easy to see. 

\begin{lemma}\label{lem:Lie_super-coalgebra}
For homogeneous elements $a, b \in \A^+$, we have
\[ 
\Delta(ab) \equiv 1 \otimes ab + ab \otimes 1 + a \otimes b + (-1)^{|a||b|}b \otimes a 
\]
modulo $\A^+ \otimes (\A^+)^2 + (\A^+)^2 \otimes \A^+$.
\end{lemma}

Set $\mathfrak{d} := \A^+/(\A^+)^2$. This is a supermodule. The \emph{Lie superalgebra}
\[
\mathfrak{g} = \mathrm{Lie}(\G)
\]
\emph{of} $\G$ is the dual supermodule $\mathfrak{d}^*$ of $\mathfrak{d}$. 
Note that $\A^*$ is the dual superalgebra of the super-coalgebra $\A$.
Regard $\mathfrak{g}$
as a super-submodule of $\A^*$  through the natural embedding
$\mathfrak{g} \subset \k \oplus \mathfrak{d}^*= (\A/(\A^+)^2)^* \subset \A^*$. By definition we have
\[
\mathfrak{g}_1=(W^{\A})^*.
\]

\begin{prop}\label{prop:Lie(G)}
The super-linear endomorphism $\mathrm{id} - c_{\A^*, \A^*}$ on $\A^* \otimes \A^*$, composed with
the product on $\A^*$, restricts to a map, $[\ , \ ] : \mathfrak{g} \otimes \mathfrak{g} \to \mathfrak{g}$, 
with which $\mathfrak{g}$ is indeed a Lie superalgebra. This satisfies (A3). 
\end{prop}
\begin{proof}
By Lemma \ref{lem:Lie_super-coalgebra} it follows that $(\mathrm{id}-c_{\A,\A})\circ \Delta$ induces a
super-linear map
\begin{equation}\label{eq:delta}
\delta : \mathfrak{d} \to \mathfrak{d}\otimes \mathfrak{d}, 
\end{equation}
which is seen to satisfy
\[
(\mathrm{id}_{\mathfrak{d} \otimes \mathfrak{d}} + c_{\mathfrak{d},\mathfrak{d}})\circ \delta = 0, \quad
(\mathrm{id}_{\mathfrak{d}\otimes \mathfrak{d}\otimes \mathfrak{d}} + 
c_{\mathfrak{d}, \mathfrak{d}\otimes \mathfrak{d}} + c_{\mathfrak{d} \otimes \mathfrak{d}, \mathfrak{d}}) 
\circ (\delta\otimes \mathrm{id}_{\mathfrak{d}})\circ \delta= 0.
\]
Therefore, $\delta$ is dualized to a map $[\ ,\ ]$ such as above, which satisfies (i), (iii) and (iv)
required to super-brackets; see Section \ref{subsec:definition_of_admissible_LSA}. 
Let $v \in \mathfrak{g}_1$. Then it follows from Lemma \ref{lem:Lie_super-coalgebra}
that given $a, b$ as in the lemma, we have
\[
v^2(ab) =  v(a)v(b) + (-1)^{|a||b|}v(b)v(a)= 0,
\]
since $v(a)v(b)=0$ unless $|a|=|b|=1$. Therefore, $v^2 \in \mathfrak{g}_0$ and $[v,v]=2v^2$. Thus 
(A3) is satisfied. The remaining (ii) is satisfied since $[[v,v],v]=2[v^2,v] = 0$. 
\end{proof}

Set $G := \G_{ev}$. Then $\overline{\A} = \O(G)$. We have the Lie algebra 
$\mathrm{Lie}(G)=(\overline{\A}^+/(\overline{\A}^+)^2)^*$ of $G$. 

\begin{lemma}\label{lem:Lie(G)g_0}
The natural embedding $\overline{\A}^*\subset \A^*$ induces an isomorphism 
$\mathrm{Lie}(G) \simeq \mathfrak{g}_0$ of Lie algebras.
\end{lemma}
\begin{proof}
One sees that this is the dual of the canonical isomorphism 
\[
\A_0^+/((\A_0^+)^2+\A_1^2) \simeq
(\A_0^+/\A_1^2)/(((\A_0^+)^2+\A_1^2)/\A_1^2).
\] 
\end{proof}

\subsection{}\label{subsec:HCP}
Let $G$ be an algebraic group. 

The Lie algebra $\mathrm{Lie}(G)$ of $G$ is naturally embedded into $\O(G)^*$, and 
the embedding gives rise to an algebra map $U(\mathrm{Lie}(G))\to \O(G)^*$. The associated pairing
\begin{equation}\label{eq:U(Lie(G))O(G)_pairing}
\langle \ , \ \rangle : U(\mathrm{Lie}(G)) \times \O(G) \to \k
\end{equation}
is a Hopf pairing. Therefore, given a left $G$-module (resp., right) structure on a $\k$-module, 
there is induced a left (resp., right) $U(\mathrm{Lie}(G))$-module structure on the $\k$-module,
as was seen in \eqref{eq:induced_supermodule}.

The right adjoint action by $G$ on itself is dualized to the right co-adjoint coaction
\begin{equation}\label{eq:coad} 
\O(G) \to \O(G) \otimes \O(G),\quad a \mapsto a_{(2)}\otimes S(a_{(1)})a_{(3)}. 
\end{equation}
This induces on $\O(G)^+/(\O(G)^+)^2$ a right $\O(G)$-comodule (or left $G$-module) structure.
We assume
\begin{itemize}
\item[(B1)] $\O(G)/(\O(G)^+)^2$ is $\k$-finite projective. 
\end{itemize} 
This is necessarily satisfied if $\k$ is a field. Under the assumption, 
the left $G$-module structure on $\O(G)^+/(\O(G)^+)^2$ just obtained is transposed to 
a right $G$-module structure on $\mathrm{Lie}(G)$.  
The induced right $U(\mathrm{Lie}(G))$-module structure coincides 
with the right adjoint action $\mathrm{ad}_r(u)(v) = [v,u]$, $u,v \in \mathrm{Lie}(G)$,
as is seen by using the fact that the pairing above satisfies
\begin{equation}\label{eq:derivation}
\langle u, ab \rangle=\langle u,a \rangle \, \varepsilon(b) + \varepsilon(a) \, \langle u,b \rangle,\quad
\langle u, S(a) \rangle = - \langle u,a \rangle
\end{equation}  
for $u \in \mathrm{Lie}(G)$, $a, b\in \O(G)$. 

Let $G$ be an algebraic group which satisfies (B1), and
let $\mathfrak{g}$ be a Lie superalgebra such that $\mathfrak{g}_0 = \mathrm{Lie}(G)$. 
Note that $\mathfrak{g}_0$ is $\k$-finite projective and so $\k$-flat; it is a right $G$-module,
as was just seen. We assume in addition,
\begin{itemize}
\item[(B2)] $\mathfrak{g}_1$ is $\k$-finite free, and $\mathfrak{g}$ is admissible, and
\item[(B3)] $\O(G)$ is $\k$-flat.
\end{itemize}
Assuming (B1) we see that (B2) is equivalent to that $\mathfrak{g}_1$ is $\k$-finite free, and $\mathfrak{g}$
satisfies (A3).  

\begin{definition}[\text{cf.~\cite[Definition 3.1]{CF}}]\label{def:HCP} 
(1)~~Suppose that the pair $(G, \mathfrak{g})$ is accompanied with a right $G$-module structure on $\mathfrak{g}_1$ 
such that the induced right $U(\mathfrak{g}_0)$-module structure coincides with
the right adjoint $\mathfrak{g}_0$-action given by \eqref{eq:right_adjoint}. 
Then $(G, \mathfrak{g})$ is called a \emph{Harish-Chandra pair} 
if the super-bracket $[\ , \ ] : \mathfrak{g}_1\otimes \mathfrak{g}_1\to \mathfrak{g}_0$ 
restricted to $\mathfrak{g}_1\otimes \mathfrak{g}_1$ is right $G$-equivariant. 

(2)~~A \emph{morphism} $(G, \mathfrak{g}) \to (G', \mathfrak{g}')$ between 
Harish-Chandra pairs is a pair $(\alpha, \beta)$ of a morphism  
$\alpha: G \to G'$ of affine groups and a Lie superalgebra map 
$\beta = \beta_0\oplus \beta_1 : \mathfrak{g} \to \mathfrak{g}'$, such that 
\begin{itemize}
\item[(i)] the Lie algebra map $\mathrm{Lie}(\alpha)$ induced from $\alpha$ coincides with $\beta_0$, and  
\item[(ii)] 
$\beta_1(v^{\gamma})= \beta_1(v)^{\alpha(\gamma)},\quad  \gamma \in G,\ v\in \mathfrak{g}_1.$
\end{itemize} 

(3)~~The Harish-Chandra pairs and their morphisms form a category $\mathsf{HCP}$. 
\end{definition}

By convention (see \eqref{eq:convention}) the equation (ii)  of (2) above should read 
\[
(\beta_1\otimes \mathrm{id}_R)((v\otimes 1)^{\gamma})
= ((\beta_1\otimes \mathrm{id}_R)(v\otimes 1))^{\alpha_R(\gamma)},
\]
where $R$ is a commutative algebra, and $\gamma \in G(R)$.

\begin{rem}\label{rem:compare_definitions}
(1)\
Suppose that $\k$ is a field of characteristic $\ne 2$. In this situation
the notion of Harish-Chandra pairs was defined by \cite[Definition ~7]{M2} in purely Hopf algebraic terms.
It is remarked by \cite[Remark 9 (2)]{M2}
that if the characteristic $\mathrm{char}\, \k$ of $\k$ is zero, 
there is a natural category anti-isomorphism between 
our $\mathsf{HCP}$ defined above and the category of the Harish-Chandra pairs as defined by
\cite[Definition~7]{M2}. But this is indeed the case without the restriction on $\mathrm{char}\, \k$. 
A key fact is the following: once we are given an algebraic group $G$, a finite-dimensional 
right $G$-module $V$ and a right
$G$-equivariant linear map $[\ , \ ] : V \otimes V \to \mathrm{Lie}(G)$, then the pair $(\O(G), V^*)$,
accompanied with $[\ , \ ]$, is a Harish-Chandra pair in the sense of \cite{M2}, if and only if the direct sum
$\mathfrak{g} := \mathrm{Lie}(G) \oplus V$ is a Lie superalgebra (in our sense), with respect to the grading
$\mathfrak{g}_0= \mathrm{Lie}(G)$, $\mathfrak{g}_1=V$, and with respect to the super-bracket which uniquely
extends (a)~the bracket on $\mathrm{Lie}(G)$, (b)~the map $[\ , \ ]$, and 
(c)~the right adjoint $\mathrm{Lie}(G)$-action on $V$ which is induced from 
the right $G$-action on $V$. See \cite[Remark 2 (1)]{M2}, but note   
that in \cite{M2}, the notion of Lie superalgebras is used in a restrictive sense 
when $\mathrm{char}\, \k =3$; 
indeed, to define the notion, the article excludes Condition (ii) 
from our axioms given in the 
beginning of Section \ref{subsec:definition_of_admissible_LSA}.  

(2)\
As the referee pointed out, our definition of Harish-Chandra pairs looks different from those definitions
given in \cite[Section 7.4]{CCF} and \cite[Section 3.1]{CF} which require that the whole super-bracket 
$[\ , \ ] : \mathfrak{g} \otimes \mathfrak{g} \to \mathfrak{g}$ is $G$-equivariant. But this follows from the weaker
requirement of ours that the restricted super-bracket $[\ , \ ]|_{\mathfrak{g}_1\otimes \mathfrak{g}_1}$  
is $G$-equivariant, since $[\ , \ ]|_{\mathfrak{g}_0\otimes \mathfrak{g}_0}$ is obviously $G$-equivariant, 
and 
$[\ , \ ]|_{\mathfrak{g}_1\otimes \mathfrak{g}_0}$ is, too, 
as will be seen below. 
Let $\gamma \in G$, $u \in \mathfrak{g}_0$ and $v \in \mathfrak{g}_1$. Note that
\[ 
\langle u, a_{(1)} \rangle \, \gamma(a_{(2)}) = \gamma(a_{(1)}) \, \langle u^{\gamma}, a_{(2)} \rangle, 
\quad a \in \mathscr{O}(G). 
\]
Then the common requirement for the induced $U(\mathfrak{g}_0)$-module structure on $\mathfrak{g}_1$ shows 
that
$
[v, u]^{\gamma} = [v^{\gamma}, u^{\gamma}].
$ 
\end{rem}

\subsection{}\label{subsec:P}
We define $\mathsf{AHSA}$ to be the full subcategory of the category of affine Hopf superalgebras
which consists of the affine Hopf superalgebras $\A$ such that 
\begin{itemize}
\item[(C1)] $\A$ is split (see Definition \ref{def:split}),
\item[(C2)] $\overline{\A}$ is $\k$-flat and 
\item[(C3)] $\overline{\A}/(\overline{\A}^+)^2$ is $\k$-finite projective. 
\end{itemize}
Note that the affinity and (C1) imply that $W^{\A}$ is $\k$-finite free. If $\k$ is a field of
characteristic $\ne 2$, then $\mathsf{AHSA}$ is precisely the category of all 
affine Hopf superalgebras. 

We define $\mathsf{ASG}$ to be the full subcategory of the category of algebraic supergroups 
which consists of the algebraic supergroups $\G$ such that $\O(\G)$ is split, and $\G_{ev}$ satisfies
(B1), (B3).  This is anti-isomorphic to $\mathsf{AHSA}$, and is
precisely the category of all algebraic supergroups if $\k$ is a field of characteristic $\ne 2$. 

Let $\G \in \mathsf{ASG}$. Set 
\[ 
\A := \O(\G), \quad G:=\G_{ev}, \quad \mathfrak{g}:=\mathrm{Lie}(\G). 
\]
Then $\A \in \mathsf{AHSA}$, and $\O(G)\, (= \overline{\A})$ satisfies (B1), (B3). 
By Proposition \ref{prop:Lie(G)}, $\mathfrak{g}$ satisfies (B2). By Lemma \ref{lem:Lie(G)g_0}
we have a natural isomorphism $\mathrm{Lie}(G) \simeq \mathfrak{g}_0$, through which we will identify
the two, and suppose $\mathfrak{g}_0 = \mathrm{Lie}(G)$. Just as was seen in \eqref{eq:coad}, 
the right co-adjoint $\overline{\A}$-coaction defined by
\begin{equation}\label{eq:coad2}
\A \to \A \otimes \overline{\A},\quad a \mapsto a_{(2)}\otimes S(\overline{a}_{(1)})\overline{a}_{(3)}, 
\end{equation}
using the notation \eqref{eq:quotient_map}, induces on $\A^+/(\A^+)^2$ a 
right $\overline{\A}$-super-comodule (or left $G$-supermodule) structure; by (C3), it is 
transposed to a right $G$-supermodule structure on $\mathfrak{g}$, which is restricted to $\mathfrak{g}_1$.  

\begin{lemma}\label{lem:P}
Given the restricted right $G$-module structure on $\mathfrak{g}_1$, the pair $(G, \mathfrak{g})$
forms a Harish-Chandra pair, and so $(G, \mathfrak{g})\in \mathsf{HCP}$. 
\end{lemma}
\begin{proof}
The right $G$-module structure on $\mathfrak{g}_1$ induces the right adjoint $\mathfrak{g}_0$-action, 
as is seen by using \eqref{eq:derivation}. 
Since one sees that the map $\delta$ given in \eqref{eq:delta} is $G$-equivariant, so is its dual, $[\ , \ ]$.  
\end{proof}

We denote this object in $\mathsf{HCP}$ by
\[
\mathbf{P}(\G)= (G, \mathfrak{g}).
\]

\begin{prop}\label{prop:P}
$\G \mapsto \mathbf{P}(\G)$ gives a functor $\mathbf{P} : \mathsf{ASG} \to \mathsf{HCP}$.
\end{prop}
\begin{proof}
Indeed, the constructions of $G$ and of $\mathfrak{g}$ are functorial. 
\end{proof}

\subsection{}\label{subsec:A(G,g)}
Let $(G, \mathfrak{g}) \in \mathsf{HCP}$. Modifying the construction of $A(C,W)$ given in \cite{M2},
we construct an object $\A(G,\mathfrak{g})$ in $\mathsf{AHSA}$. To be close to \cite{M2} for notation we set
\[
J := U(\mathfrak{g}_0),\quad C := \O(G),\quad W := \mathfrak{g}_1^*. 
\]
Then $W$ is $\k$-finite free. It is a right $C$-comodule, or a left $G$-module, 
with the right $G$-module structure on $\mathfrak{g}_1$ transposed to $W$. 

Let $\mathbb{N}=\{ 0,1,2,\dots \}$ denote the semigroup of non-negative integers. 
A supermodule is said to be $\mathbb{N}$-\emph{graded},  
if it is $\mathbb{N}$-graded as a $\k$-module and
if the original $\mathbb{Z}_2$-grading equals the $\mathbb{N}$-grading modulo $2$.
A Hopf superalgebra
is said to be $\mathbb{N}$-\emph{graded} \cite[Definition 1]{M2}, 
if it is $\mathbb{N}$-graded as an algebra and coalgebra and
if the original $\mathbb{Z}_2$-grading equals the $\mathbb{N}$-grading modulo $2$.

Recall from Section \ref{subsec:Hopf_pairing} that the tensor algebra 
$\T(\mathfrak{g}_1)=\bigoplus_{n=0}^{\infty}\T^n(\mathfrak{g}_1)$ 
on $\mathfrak{g}_1$ is a cocommutative Hopf
superalgebra; this is $\mathbb{N}$-graded. Recall that $\mathfrak{g}_0$ 
acts on $\mathfrak{g}_1$ by the right
adjoint; see \eqref{eq:right_adjoint}. 
This uniquely extends to a right $J$-module-algebra structure on $\T(\mathfrak{g}_1)$,
with which is associated the smash-product algebra \cite[p.155]{Sw} 
\[
\mathscr{H} := J \mltimes \T(\mathfrak{g}_1). 
\] 
Given the tensor-product coalgebra structure on $J \otimes \T(\mathfrak{g}_1)$, 
this $\mathscr{H}$ is a cocommutative Hopf superalgebra, which is $\mathbb{N}$-graded so that  
$\mathscr{H}(n) = J \otimes \T^n(\mathfrak{g}_1)$,\ $n \in \mathbb{N}$; see \cite[Section 3.2]{M2}.
Set 
\[ 
\U := \U(\mathfrak{g}). 
\]
Since we see that $\mathscr{H}$ is the quotient Hopf superalgebra of $\T(\mathfrak{g})$ divided by
the Hopf super-ideal generated by
\[
zw - wz - [z, w],\quad z \in \mathfrak{g},\ w \in \mathfrak{g}_0, 
\]
it follows that $\U= \mathscr{H}/\mathscr{I}$, where $\mathscr{I}$ is the Hopf super-ideal 
of $\mathscr{H}$ generated by the even primitives
\begin{equation}\label{eq:generators}
1 \otimes (uv+vu) - [u,v]\otimes 1, \ 1\otimes v^2- \frac{1}{2} [v,v] \otimes 1, 
\end{equation}
where $u, v \in \mathfrak{g}_1$.

Let $\T_c(W)$ denote the \emph{tensor coalgebra} on $W$, as given in \cite[Section 4.1]{M2}; this is a commutative
$\mathbb{N}$-graded Hopf superalgebra. In fact, this equals the tensor algebra 
$\T(W) = \bigoplus_{n=0}^{\infty}\T^n(W)$
as an $\mathbb{N}$-graded module, and is the \emph{graded dual} $\bigoplus_{n=0}^{\infty}\T^n(\mathfrak{g}_1)^*$
of $\T(\mathfrak{g}_1)$ (see \cite[p.231]{Sw}) as an algebra and coalgebra. 
Suppose that $\T^0(W)=\k$ is the trivial right $C$-comodule, and
$\T^n(W)$, $n > 0$, is the $n$-fold tensor product of the right $C$-comodule $W$. 
Then $\T_c(W)$ turns into a right $C$-comodule coalgebra.
The associated smash coproduct $C \cmdblackltimes \T_c(W)$,  
given the tensor-product algebra structure on $C \otimes \T_c(W)$, 
is a commutative $\mathbb{N}$-graded Hopf superalgebra. 
Explicitly, the coproduct is given by
\begin{equation}\label{eq:smash_coproduct}
\Delta(c \otimes d) = \big(c_{(1)} \otimes (d_{(1)})^{(0)}\big)\otimes 
\big((d_{(1)})^{(1)}c_{(2)}\otimes d_{(2)}\big),
\end{equation}
where $c \in C$,\ $d \in \T_c(W)$, and $d \mapsto d^{(0)} \otimes d^{(1)}$ denotes the right $C$-comodule
structure on $\T_c(W)$. 

In general, given an $\mathbb{N}$-graded supermodule $\mathscr{A}= \bigoplus_{n=0}^{\infty}\mathscr{A}(n)$, we suppose that 
it is given the linear topology defined by the the descending chains of super-ideals
\[
\bigoplus_{i>n}\mathscr{A}(n),\ n=0,1,\dots.
\]
The completion $\widehat{\mathscr{A}}$ coincide with the direct product $\prod_{n=0}^{\infty}\mathscr{A}(n)$. 
This is not $\mathbb{N}$-graded any more, but is still a supermodule. Given another 
$\mathbb{N}$-graded supermodule $\mathscr{B}$, the tensor product 
$\mathscr{A} \otimes \mathscr{B}$ is naturally an $\mathbb{N}$-graded supermodule. 
The complete tensor product $\widehat{\mathscr{A}}\, \widehat{\otimes}\, \widehat{\mathscr{B}}$
coincides with the completion of $\mathscr{A} \otimes \mathscr{B}$. We regard $\k$ as a trivially 
$\mathbb{N}$-graded supermodule, which is discrete. Suppose that $\mathscr{A}$ is an
$\mathbb{N}$-graded Hopf superalgebra. 
The structure maps on $\mathscr{A}$, being $\mathbb{N}$-graded and hence continuous, are completed to
\[
\widehat{\Delta} : \widehat{\mathscr{A}}\longrightarrow \widehat{\mathscr{A}}\mathbin{\widehat{\otimes}} \widehat{\mathscr{A}},\quad
\widehat{\varepsilon} : \widehat{\mathscr{A}}\longrightarrow \Bbbk,\quad
\widehat{S} : \widehat{\mathscr{A}}\longrightarrow \widehat{\mathscr{A}}.
\]
Satisfying the axiom of Hopf superalgebras with $\otimes$ 
replaced by $\mathbin{\widehat{\otimes}}$, this $\widehat{\mathscr{A}}$ may be called a 
\emph{complete topological Hopf superalgebra}. If $\mathscr{A}$ is commutative, then $\widehat{\mathscr{A}}$
is, too. See \cite[Section 2.3]{M2}. 

Applying the construction above to $C \cmdblackltimes \T_c(W)$, we suppose
\[
\mathscr{A}= C \cmdblackltimes \T_c(W),\quad \widehat{\mathscr{A}}=\prod_{n=0}^{\infty}C\otimes \T^n(W) 
\]
in what follows. We let 
\begin{equation}\label{eq:proj_to_C}
\pi : \widehat{\mathscr{A}}\to C\otimes \T^0(W) = C 
\end{equation}
denote the natural projection.  

We regard $C$ as a left $J$-module by
\[ 
x c := c_{(1)}\, \langle x, c_{(2)}\rangle,\quad x \in J,\ c \in C,
\]
where $\langle \ , \ \rangle : J \times C\to \k$ denotes the canonical 
Hopf pairing; see \eqref{eq:U(Lie(G))O(G)_pairing}. 

Let $\mathrm{Hom}_J$ denote the $\k$-module of left $J$-module maps. We
regard $\mathrm{Hom}_J(\mathscr{H}, C)$ as the completion of the $\mathbb{N}$-graded supermodule 
$\bigoplus_{n=0}^{\infty} \operatorname{Hom}_J(J \otimes \T^n(\mathfrak{g}_1),C)$. 
The canonical isomorphisms
\begin{equation}\label{eq:canonical_isom}
C \otimes \T^n(W)=\mathrm{Hom}(\T^n(\mathfrak{g}_1),C) \isomto 
\operatorname{Hom}_J(J \otimes \T^n(\mathfrak{g}_1),C),\quad n=0,1,\dots
\end{equation}
altogether amount to a super-linear homeomorphism
\begin{equation}\label{eq:xi1} 
\xi : \widehat{\mathscr{A}} \overset{\simeq}{\longrightarrow} \operatorname{Hom}_J(\mathscr{H}, C). 
\end{equation}

Tensoring the canonical pairings 
$J \times C\to \k$ and $\T(\mathfrak{g}_1) \times \T_c(W)\to \k$, we define 
\begin{equation}\label{eq:HA_pairing}
\langle \ , \ \rangle : \mathscr{H}\times \mathscr{A} \to \k,\quad \langle x\otimes y , c \otimes d \rangle
=\langle x, c\rangle \, \langle y,d \rangle,
\end{equation}
where $x\in J$,\ $y \in \T(\mathfrak{g}_1)$,\ $c \in C$,\ $d \in \T_c(W)$.  
This is a Hopf pairing, as was seen in \cite[Proposition 17]{M2}.

\begin{lemma}\label{lem:xi} 
$\xi$ is determined by
\begin{equation}\label{eq:xi2} 
\xi(a)(x) = \pi(a_{(1)})\, \langle x, a_{(2)}\rangle , \quad a \in \mathscr{A},\ 
x \in \mathscr{H}. 
\end{equation}    
\end{lemma}
\begin{proof}
Note that if $a = c \otimes d$, where $c \in C$, $d \in \T_c(W)$, then 
\[
\pi(a_{(1)})\otimes a_{(2)} = c_{(1)} \otimes (c_{(2)} \otimes d). 
\]
Then the lemma follows since $\xi$ is the completion of the $\mathbb{N}$-graded linear map
\[ \mathscr{A} = C \otimes (\bigoplus_{n=0}^{\infty} \T^n(W)) \to
\mathrm{Hom}_J(J \otimes (\bigoplus_{n=0}^{\infty} \T^n(\mathfrak{g}_1)), C)
\]
given by
$ c \otimes d \mapsto \big(x \otimes y \mapsto x c \, \langle y, d\rangle \big), $
and this last element equals $c_{(1)}\, \langle x\otimes y, c_{(2)}\otimes d \rangle$.  
\end{proof}
\begin{rem}\label{rem:xi}
Recall that $\langle \mathscr{H}(n), \mathscr{A}(m)\rangle=0$ unless $n=m$. Therefore, 
the pairing \eqref{eq:HA_pairing} uniquely extends to 
\begin{equation}\label{eq:HhatA_pairing}
\langle \ , \ \rangle : \mathscr{H}\times \widehat{\mathscr{A}} \to \k 
\end{equation}
so that for each $x \in \mathscr{H}$, $\langle x, \ \rangle : \widehat{\mathscr{A}} \to \k$ is continuous.
Using this pairing one sees that the value $\xi(a)$ at $a \in \widehat{\mathscr{A}}$ is given by the same 
formula as \eqref{eq:xi2}, with $\pi(a_{(1)}) \otimes a_{(2)}$ understood to be 
$(\pi \, \widehat{\otimes} \, \mathrm{id})\circ \widehat{\Delta}(a)$. 
\end{rem}

We aim to transfer the structures on 
$\widehat{\mathscr{A}}$ to $\operatorname{Hom}_J(\mathscr{H}, C)$
through $\xi$; see Proposition \ref{prop:transferred_structures} below. 

Recall from Section \ref{subsec:P} that $\mathfrak{g}_0$ is a right $G$-module. Combined with 
the given right $G$-module structure on $\mathfrak{g}_1$, it results that $\mathfrak{g}\in 
\mathsf{SMod}\text{-}G$; see \eqref{eq:G-SMod}. Moreover, $\mathfrak{g}$ is a Lie-algebra object in
$\mathsf{SMod}\text{-}G$,
since the super-bracket $[\ , \ ] : \mathfrak{g}\otimes \mathfrak{g}\to \mathfrak{g}$ is $G$-equivariant,
as was proved in Remark \ref{rem:compare_definitions} (2).   

We regard $\mathscr{A}$ as a right $C$-super-comodule, or an object in $G\text{-}\mathsf{SMod}$, 
with respect to the right co-adjoint coaction
\begin{equation}\label{eq:coad3}
\mathscr{A}\to \mathscr{A}\otimes C,\quad a \mapsto a_{(2)}\otimes S(\pi(a_{(1)}))\, \pi(a_{(3)}). 
\end{equation}

\begin{lemma}\label{lem:extended_G-action}
We have the following.
\begin{itemize}
\item[(1)]
The right $G$-supermodule structure on $\mathfrak{g}$ uniquely extends 
to that on $\mathscr{H}$ so that $\mathscr{H}$ turns into an algebra object in $\mathsf{SMod}\text{-}G$. 
In fact, $\mathscr{H}$ turns into a Hopf-algebra object in $\mathsf{SMod}\text{-}G$.
\item[(2)]
With the structure above, $\mathscr{A}$ turns into a Hopf-algebra object in $G\text{-}\mathsf{SMod}$. 
\item[(3)]
The resulting structures are dual to each other in the sense that
\begin{equation}\label{eq:pairing_G-action}
\langle x^{\gamma}, a \rangle = \langle x, {}^{\gamma}a \rangle,\quad
\gamma \in G,\ x \in \mathscr{H},\ a \in \mathscr{A}.
\end{equation}
\end{itemize}
\end{lemma}
\begin{proof}
(1)~~The right $G$-supermodule structure on $\mathfrak{g}$ uniquely extends to that 
on $\T(\mathfrak{g})$ so that $\T(\mathfrak{g})$ turns into an algebra object in $\mathsf{SMod}\text{-}G$. 
The extended structure factors to $\mathscr{H}$, since we have 
$[z,w]^{\gamma}=[z^{\gamma}, w^{\gamma}]$, where 
$\gamma \in G$, $z \in \mathfrak{g}$ and $w \in \mathfrak{g}_0$. 
One sees easily that the resulting structure on $\mathscr{H}$ is such as mentioned above. 

(2)~~This is easy to see.

(3)~~Let $a \in C$, and let $x=u_1\dots u_r$ be an element of $J$ with $u_i \in \mathfrak{g}_0$.  
One sees by induction on $r$ that \eqref{eq:pairing_G-action} holds for these $x$ and $a$,
using the fact that $G$-actions preserve the algebra structure on $J$ and the coalgebra structure on $C$.

We see from \eqref{eq:smash_coproduct} that the left $G$-module structure on $\mathscr{A}$, 
restricted to $\T_c(W)= \k \otimes \T_c(W)$, is
precisely what corresponds to the original right $C$-comodule structure on $\T_c(W)$. 
It follows that \eqref{eq:pairing_G-action} holds for $x \in \T(\mathfrak{g}_1)$, $a \in \T_c(W)$. 

The desired equality now follows from the definition \eqref{eq:HA_pairing} together with the fact
that the $G$-actions preserve the products on $\mathscr{H}$ and on $\mathscr{A}$.  
\end{proof}

For each $n \ge 0$ we have a natural linear isomorphism (see \eqref{eq:canonical_isom}) from 
\[ 
\bigoplus_{i+j=n} \mathrm{Hom}_J(J\otimes \T^i(\mathfrak{g}_1), C)\otimes 
\mathrm{Hom}_J(J\otimes \T^j(\mathfrak{g}_1), C)
\]
onto the $\k$-module  
\[ 
\bigoplus_{i+j=n} \mathrm{Hom}_{J\otimes J}((J\otimes \T^i(\mathfrak{g}_1))\otimes 
(J\otimes \T^j(\mathfrak{g}_1)), C\otimes C) 
\]
which consists of left $J\otimes J$-module maps. 
The direct product $\prod_{n=0}^{\infty}$ of the isomorphisms gives the super-linear homeomorphism 
\begin{equation*}\label{eq:Hom_homo} 
\mathrm{Hom}_J(\mathscr{H}, C)\, \widehat{\otimes}\, \mathrm{Hom}_J(\mathscr{H}, C)
\overset{\approx}{\longrightarrow} \mathrm{Hom}_{J\otimes J}(\mathscr{H}\otimes \mathscr{H}, C\otimes C),
\end{equation*}
which is indeed the completion of the continuous map  
$f \otimes g \mapsto (x \otimes y \mapsto f(x)\otimes g(y))$,
where $f, g \in \mathrm{Hom}_J(\mathscr{H}, C)$,\ $x, y \in \mathscr{H}$. 
This homeomorphism will be used in Part 2 below.

\begin{prop}\label{prop:transferred_structures}
Suppose that
$f, g \in \operatorname{Hom}_J(\mathscr{H}, C),\ x, y \in \mathscr{H}$ and $\gamma, \delta \in G(R)$,
where $R$ is an arbitrary commutative algebra.

\begin{itemize}
\item[(1)]
The product, the identity, the counit $\widehat{\varepsilon}$ and the antipode $\widehat{S}$ on 
$\widehat{\mathscr{A}}$ are transferred to $\operatorname{Hom}_J(\mathscr{H},C)$ through $\xi$ so that 
\begin{align*}
fg(x) &= f(x_{(1)})g(x_{(2)}), \\ 
\xi(1)(x) &= \varepsilon (x)1,\\
\widehat{\varepsilon}(f) &= \varepsilon(f(1)),\\
\langle \gamma,\ \widehat{S}(f)(x)\rangle &= \langle \gamma^{-1},\ f(S(x)^{\gamma^{-1}}) \rangle.
\end{align*}
\item[(2)]
Through $\xi$ and $\xi ~ \widehat{\otimes} ~ \xi$,  the coproduct on $\widehat{\mathscr{A}}$ 
is translated to 
\begin{align*}
\widehat{\Delta} : \operatorname{Hom}_J(\mathscr{H}, C) \to 
&\operatorname{Hom}_J(\mathscr{H}, C) ~ \widehat{\otimes} ~ 
\operatorname{Hom}_J(\mathscr{H}, C)\\
&\approx \operatorname{Hom}_{J \otimes J}(\mathscr{H}\otimes \mathscr{H}, C\otimes C)
\end{align*}
so that 
\begin{equation*}
\langle (\gamma, \delta),\ \widehat{\Delta}(f)(x \otimes y) \rangle 
= \langle \gamma \delta,\ f(x^{\delta}\,  y) \rangle. 
\end{equation*}
\end{itemize} 
Here, $\langle \gamma^{\pm1},\ \rangle$,\ $\langle \gamma \delta,\ \rangle$ and
$\langle (\gamma, \delta),\ \rangle$ denote the functor points in $G(R)$ and in $(G\times G)(R)$, 
respectively. 
\end{prop}

The formulas are essentially the same as those given in \cite[Proposition 18 (2), (3)]{M2}. 
One will see below that the proof here, using Lemma \ref{lem:xi}, is simpler.  

\begin{proof}
(1)~~Let $a \in \mathscr{A}$, and write as $\pi(a)= \overline{a}$. Then one has
\begin{equation}\label{eq:gamma_a}
{}^{\gamma}a = \langle \gamma^{-1}, \overline{a}_{(1)}\rangle\, a_{(2)}\, 
\langle \gamma, \overline{a}_{(3)}\rangle,\quad \gamma \in G.
\end{equation}
To prove the last formula we may suppose $f =\xi(a)$, since we evaluate $f$, $\widehat{S}(f)$
on $\mathscr{H}$. By using Lemma \ref{lem:xi} we see that
\begin{align*} 
\text{LHS} &=  \langle x , S(a_{(1)}) \rangle\, \langle \gamma, S(\overline{a}_{(2)})\rangle
= \langle S(x), a_{(1)} \rangle\, \langle \gamma^{-1}, \overline{a}_{(2)} \rangle \\
&= \langle \gamma^{-1} , \overline{a}_{(1)} \rangle\, \langle \gamma , \overline{a}_{(2)} \rangle\, 
\langle S(x), a_{(3)} \rangle\, \langle \gamma^{-1}, \overline{a}_{(4)} \rangle\\
&= \langle \gamma^{-1} , \overline{a}_{(1)}\rangle\, \langle S(x), {}^{\gamma^{-1}}a_{(2)} \rangle 
=\text{RHS}.
\end{align*}
The rest is easy to see.

(2)~~As above we may suppose $f=\xi(a)$, $a \in \mathscr{A}$. Then
\begin{align*} 
\text{LHS} 
&= \langle \gamma, \overline{a}_{(1)} \rangle\, \langle x, a_{(2)}\rangle\, 
\langle \delta, \overline{a}_{(3)} \rangle\,
\langle y, a_{(4)}\rangle\\
&= \langle \gamma, \overline{a}_{(1)} \rangle\, \langle \delta, \overline{a}_{(2)} \rangle\,
\langle \delta, S(\overline{a}_{(3)})\rangle\, 
\langle x, a_{(4)}\rangle\,
\langle \delta, \overline{a}_{(5)}\rangle\, 
\langle y, a_{(6)}\rangle\\
&= \langle \gamma, \overline{a}_{(1)} \rangle\, \langle \delta, \overline{a}_{(2)} \rangle\, 
\langle x, {}^{\delta}a_{(3)}\rangle\, \langle y, a_{(4)}\rangle 
=\text{RHS}.
\end{align*} 
\end{proof} 

Recall from \eqref{eq:generators} that $\mathscr{I}$ is the Hopf super-ideal of $\mathscr{H}$ such that
$\mathscr{H}/\mathscr{I} = \U$. Note that by the $\k$-flatness assumption (B3), the following statement
makes sense. 

\begin{lemma}\label{lem:G-stable}
$\mathscr{I}$ is $G$-stable, or in other words, it is $C$-costable. Therefore,
$\U \in \mathsf{SMod}\text{-}G$. 
\end{lemma}
\begin{proof}
Since $[ \ , \ ] : \mathfrak{g}_1 \otimes \mathfrak{g}_1 \to \mathfrak{g}_0$ is $G$-equivariant, 
it follows that the elements $uv + vu -[u,v]$ from 
\eqref{eq:generators} generate in $\mathscr{H}$ a $C$-costable $\k$-submodule. 

Let $\rho : \mathscr{H} \to C\otimes \mathscr{H}$ be the left $C$-comodule structure on $\mathscr{H}$.
Let $v \in \mathfrak{g}_1$, and suppose $\rho(v) = \sum_i c_i\otimes v_i$. By (B3),
$C \otimes \mathfrak{g}_0$ is $2$-torsion free. Therefore, we can conclude that
\begin{equation}\label{eq:verify(F5)}
\rho(\frac{1}{2}[v,v]) = \sum_i c_i^2\otimes\frac{1}{2}[v_i,v_i]+\sum_{i<j} c_ic_j \otimes [v_i,v_j], 
\end{equation}
by seeing that the doubles of both sides coincide. It follows that 
\[  
\rho(v^2-\frac{1}{2}[v,v]) =\sum_ic_i^2\otimes (v_i^2-\frac{1}{2}[v_i,v_i])+ 
\sum_{i<j} c_ic_j\otimes (v_iv_j+v_jv_i-[v_i,v_j]).
\]
Since this is contained in $C \otimes\mathscr{I}$, the lemma follows. 
\end{proof}

Since $\mathfrak{g}$ is admissible, it follows by Corollary \ref{cor:co-split}
that there is a unit-preserving
left $J$-module super-coalgebra isomorphism
\begin{equation}\label{eq:phi}
\phi : J \otimes \wedge(\mathfrak{g}_1) \overset{\simeq}{\longrightarrow} \U.
\end{equation}
We fix this $\phi$ for use in what follows.

\begin{corollary}\label{cor:Hom_J(U,C)}
$\mathrm{Hom}_J(\U,C)$ is a discrete super-subalgebra of $\mathrm{Hom}_J(\mathscr{H},C)$, and is
stable under $\widehat{S}$. Moreover, the map $\widehat{\Delta}$ given in 
Proposition \ref{prop:transferred_structures}~(2) sends $\mathrm{Hom}_J(\U,C)$ into
$\mathrm{Hom}_{J\otimes J}(\U \otimes \U,C\otimes C)$.  
\end{corollary}
\begin{proof}
Since $\U$ is finitely generated as a left $J$-module by \eqref{eq:phi}, we have 
$\mathrm{Hom}_J(\U, C)\subset \mathrm{Hom}_J(J \otimes \big(\bigoplus_{i<n}\T^i(\mathfrak{g}_1)\big), C)$
for $n$ large enough. This means that $\mathrm{Hom}_J(\U,C)$ is discrete. The rest follows easily from
Lemma \ref{lem:G-stable}. 
\end{proof} 

Given a Harish-Chandra pair $(G, \mathfrak{g})$ as above, we define 
\[
\A(G, \mathfrak{g})
\] 
to be the $\k$-submodule of $\widehat{\mathscr{A}}$ such that the homeomorphism 
$\xi$ given in \eqref{eq:xi1} restricts to a linear isomorphism
\begin{equation}\label{eq:A(G,g)Hom_J(U,C)}
\eta : \A(G, \mathfrak{g})\overset{\simeq}{\longrightarrow} \mathrm{Hom}_J(\U, C).
\end{equation} 
In what follows we set $\A :=\A(G, \mathfrak{g})$. 

\begin{lemma}\label{lem:A(G,g)}
We have the following. 
\begin{itemize}
\item[(1)] $\A$ is a discrete super-subalgebra of $\widehat{\mathscr{A}}$, which is stable under 
$\widehat{S}$.
\item[(2)] The canonical map 
$\A \otimes \A \to \widehat{\mathscr{A}}\, \widehat{\otimes}\, \widehat{\mathscr{A}}$ is an injection.
Regarding this injection as an inclusion, we have $\widehat{\Delta}(\A) \subset \A \otimes \A$.
\item[(3)] $(\A, \widehat{\Delta}|_{\A}, \widehat{\varepsilon}|_{\A}, \widehat{S}|_{\A})$ 
is a commutative Hopf superalgebra. 
\end{itemize}
\end{lemma}
\begin{proof}
(1)\ This follows from Corollary \ref{cor:Hom_J(U,C)}.

(2)\ By using $\eta$, the canonical map above is identified with the composite of the canonical map
\begin{equation}\label{equ:etacano}
\mathrm{Hom}_J(\mathbf{U}, C) \otimes \mathrm{Hom}_J(\mathbf{U}, C)
\to \mathrm{Hom}_{J\otimes J}(\mathbf{U}\otimes \mathbf{U}, C\otimes C)
\end{equation}
with the embedding
$
\mathrm{Hom}_{J\otimes J}(\mathbf{U}\otimes \mathbf{U}, C\otimes C) \subset
\mathrm{Hom}_{J\otimes J}(\mathscr{H}\otimes \mathscr{H}, C\otimes C). 
$
By using $\phi$, the map \eqref{equ:etacano} is identified with the canonical map
\[
\mathrm{Hom}(\wedge(\mathfrak{g_1}), C) \otimes \mathrm{Hom}(\wedge(\mathfrak{g_1}), C)
\to \mathrm{Hom}(\wedge(\mathfrak{g_1})\otimes \wedge(\mathfrak{g_1}), C\otimes C),
\] 
which is an isomorphism 
since $\wedge(\mathfrak{g_1})$ is $\k$-finite free. This proves the desired injectivity.
The rest follows from Corollary \ref{cor:Hom_J(U,C)}.

(3)\ Just as above the canonical map 
$\A \otimes \A \otimes \A \to 
\widehat{\mathscr{A}}\, \widehat{\otimes}\, \widehat{\mathscr{A}}\, \widehat{\otimes}\, \widehat{\mathscr{A}}$ 
is seen to be an injection. From this we see that $\widehat{\Delta}|_{\A}$ is coassociative. 
The rest is easy to see.
\end{proof}

The restriction $\pi|_{\A}$ of the projection \eqref{eq:proj_to_C} to $\A$ is 
a Hopf superalgebra map, which we denote by
\begin{equation}\label{eq:A_to_C}
\A \to C,\quad a\mapsto \overline{a}.
\end{equation} 
This notation is consistent with \eqref{eq:quotient_map}, as will be seen from Lemma
\ref{lem:psi} (2). 
We see from Remark \ref{rem:xi} that the pairing \eqref{eq:HhatA_pairing} induces 
\begin{equation}\label{eq:UA_pairing}
\langle \ , \ \rangle : \U \times \A \to \k,
\end{equation}
and the following lemma holds.

\begin{lemma}\label{lemma:eta}
$\eta$ is given by essentially the same formula as \eqref{eq:xi2} so that
\[
\eta(a)(x) = \overline{a}_{(1)}\, \langle x, a_{(2)} \rangle, \quad a \in \A,\ x \in \U.
\]
\end{lemma}

Define a map $\varrho$ to be the composite
\begin{equation}\label{eq:varrho}
\varrho : \A \overset{\eta}{\longrightarrow} \mathrm{Hom}_J(\U , C) \simeq
\mathrm{Hom}(\wedge(\mathfrak{g}_1) , C) \overset{\varepsilon_*}{\longrightarrow}
\wedge(\mathfrak{g}_1)^* = \wedge(W),
\end{equation}
where the second isomorphism is the one induced from the fixed $\phi$ (see \eqref{eq:phi}), and 
the following $\varepsilon_*$ denotes $\mathrm{Hom}(\wedge(\mathfrak{g}_1),\varepsilon)$.

\begin{lemma}\label{lem:psi}
We have the following.
\begin{itemize}
\item[(1)] The map 
\begin{equation*}\label{eq:psi}
\psi : \A \to C \otimes \wedge(W),\quad \psi(a) = \overline{a}_{(1)}\otimes \varrho(a_{(2)}) 
\end{equation*}
is a counit-preserving isomorphism of left $C$-comodule superalgebras. 
\item[(2)] We have natural isomorphisms
\begin{equation}\label{eq:two_isoms}
\overline{\A}\simeq C, \quad W^{\A} \simeq W=\mathfrak{g}_1^*
\end{equation}
of Hopf algebras and of $\k$-modules, respectively.
\end{itemize}
\end{lemma}
\begin{proof}
(1)\ Compose the isomorphism 
$\mathrm{Hom}_J(\U,C)\simeq \mathrm{Hom}(\wedge(\mathfrak{g}_1),C)$ in \eqref{eq:varrho} 
with the canonical one $\mathrm{Hom}(\wedge(\mathfrak{g}_1),C)\simeq C \otimes \wedge(W)$.
Through the composite we will identify as $\mathrm{Hom}_J(\U,C)=C \otimes \wedge(W)$.
Since $\langle x, a \rangle = \varepsilon(\eta(a)(x))$, $a\in \A$, $x\in \U$, one sees that 
$\psi$ is identified with $\eta$, whence it is a bijection. 
The desired result follows since $\varrho$ is a counit-preserving superalgebra map. 

(2)\ We see from the isomorphism just obtained that the Hopf superalgebra map \eqref{eq:A_to_C} induces
$\overline{\A} \simeq C$, and the pairing \eqref{eq:UA_pairing}, 
restricted to $\mathfrak{g}_1 \times \A$, induces $W^{\A} \simeq \mathfrak{g}_1^*$.
\end{proof}

The lemma shows the following. 

\begin{prop}\label{prop:A(G,g)}
$\A(G, \mathfrak{g}) \in \mathsf{AHSA}$.
\end{prop}

We let
\[
\G(G, \mathfrak{g})
\]
denote the object in $\mathsf{ASG}$ which corresponds to $\A(G, \mathfrak{g})$. 

\begin{prop}\label{prop:functor_from_(G,g)}
$(G, \mathfrak{g})\mapsto \G(G, \mathfrak{g})$ gives a functor $\G : \mathsf{HCP} \to 
\mathsf{ASG}$. 
\end{prop}
\begin{proof}
This follows since the constructions of $\widehat{\mathscr{A}}$, $\mathrm{Hom}_J(\mathscr{H}, C)$
and $\mathrm{Hom}_J(\U, C)$ are all functorial, and the homeomorphism $\xi$ is natural. 
\end{proof}

\begin{prop}\label{prop:PG=id}
The Harish-Chandra pair $\mathbf{P}(\G(G, \mathfrak{g}))$ 
associated with $\G(G, \mathfrak{g})$ is naturally isomorphic to the
original $(G, \mathfrak{g})$. 
\end{prop}

To prove this we need a lemma. Set $\A :=\A(G, \mathfrak{g})$, again. Then $\A$ is an object 
(indeed, a Hopf-algebra object) in $G\text{-}\mathsf{SMod}$, being defined by the same formula 
as \eqref{eq:gamma_a}. Recall from Lemma \ref{lem:G-stable} that 
$\mathbf{U} \in \mathsf{SMod}\text{-}G$. 

\begin{lemma}\label{lem:UA_pairing}
The pairing \eqref{eq:UA_pairing} is a Hopf pairing such that
\begin{equation}\label{eq:UA_formula}
\langle x^{\gamma}, a \rangle = \langle x, {}^{\gamma}a \rangle,\quad x\in \U,\ a \in \A. 
\end{equation}
\end{lemma}
\begin{proof}
Note that the co-adjoint coaction $\mathscr{A} \to \mathscr{A}\otimes C$ given in \eqref{eq:coad3}
is completed to 
$\widehat{\mathscr{A}} \to \widehat{\mathscr{A}} \otimes C$, by which $\widehat{\mathscr{A}}$ is a
left $G$-supermodule including $\A$ as a $G$-super-submodule. One sees that the pairing \eqref{eq:HhatA_pairing}
satisfies the same formula as \eqref{eq:pairing_G-action} for $a \in \widehat{\mathscr{A}}$. 
The resulting formula shows \eqref{eq:UA_formula}.

The rest follows since the pairing \eqref{eq:HhatA_pairing} satisfies the formulas 
\eqref{eq:Hopf_pairing_conditions}
required to Hopf pairings.
Here we understand that for $x, y \in \mathscr{H}$ and $a \in \widehat{\mathscr{A}}$, 
$\langle x, a_{(1)} \rangle \, \langle y, a_{(2)} \rangle$ 
represents $\langle x \otimes y, \widehat{\Delta}(a) \rangle$; this last denotes the pairing on 
$(\mathscr{H} \otimes \mathscr{H}) \times (\widehat{\mathscr{A}}\, \widehat{\otimes}\, \widehat{\mathscr{A}})$
which is obtained naturally from the pairing on 
$(\mathscr{H} \otimes \mathscr{H}) \times (\mathscr{A} \otimes \mathscr{A})$, just as \eqref{eq:HhatA_pairing} 
is obtained from \eqref{eq:HA_pairing}.  
\end{proof}

\begin{proof}[Proof of Proposition \ref{prop:PG=id}]
We see from the definition of $\psi$ that the pairing $\langle \ , \ \rangle : \U \times \A \to \k$ 
given in \eqref{eq:UA_pairing} satisfies  
\[
\langle \phi(x \otimes y), a \rangle = 
\langle x, \overline{a}_{(1)}\rangle \, \langle y, \varrho(a_{(2)}) \rangle,
\quad x\in J, \ y \in \wedge(\mathfrak{g}_1), \ a \in \A, 
\]
What appear on the right-hand side are the canonical pairings on $J \times C$ and on 
$\wedge(\mathfrak{g}_1) \times \wedge(W)$.  
It follows that the pairing induces 
a non-degenerate pairing $ \mathfrak{g} \times \A^+/(\A^+)^2 \to \k$. 
Lemma \ref{lem:UA_pairing} shows that the last pairing induces an isomorphism $\mathrm{Lie}(\G) \simeq \mathfrak{g}$
of Lie superalgebras, where $\G := \G(G, \mathfrak{g})$. 
In addition, the isomorphism $W^{\A} \simeq \mathfrak{g}_1^*$ 
obtained in \eqref{eq:two_isoms}
is indeed $G$-equivariant. It follows that the Lie superalgebra isomorphism together with 
$\overline{\A} \simeq C$ give 
the desired isomorphism of Harish-Chandra pairs. It is natural since the construction of \eqref{eq:UA_pairing}
is functorial. 
\end{proof}

\begin{rem}\label{rem:construction_generalized}
One sees that the construction above gives an affine (not necessarily algebraic) supergroup, 
more generally, starting with a pair $(G,\mathfrak{g})$ such that 
\begin{itemize}
\item[(i)] $G$ is an affine group with $\O(G)$ $\k$-flat,
\item[(ii)] $\mathfrak{g}$ is an admissible Lie superalgebra with $\mathfrak{g}_1$ $\k$-finite (free),
\item[(iii)] $\mathfrak{g}$ is given a right $G$-supermodule structure such that 
the super-bracket on $\mathfrak{g}$ is $G$-equivariant, 
and 
\item[(iv)] there is given a bilinear map $\langle \ , \ \rangle : \mathfrak{g}_0 \times \O(G) \to \k$ such that
\begin{itemize}
\item[] $\langle x, ab \rangle = \langle x, a \rangle \, \varepsilon(b)+ \varepsilon(a) \langle x,b \rangle$,
\item[] $\langle x^{\gamma}, a \rangle = \langle x, {}^{\gamma}a \rangle$,
\item[] $[z, x] = \langle x, z^{(-1)}\rangle \, z^{(0)}$, 
\end{itemize}
where $x \in \mathfrak{g}_0$,\ $ a,b  \in \O(G)$,\ $\gamma \in G$,\ $z \in \mathfrak{g}$,\
and $z \mapsto z^{(-1)}\otimes z^{(0)}$ denotes the left $\mathscr{O}(G)$-super-comodule structure on $\mathfrak{g}$ 
which corresponds to the given right $G$-supermodule structure. 
\end{itemize}
Here we do not assume (B1) or $G$ being algebraic. Given a super Lie group, say $\mathfrak{G}$, 
we have in mind as $G$ and $\mathfrak{g}$ above, the universal algebraic hull of the associated Lie group
$\mathfrak{G}_{red}$ and the Lie superalgebra $\mathrm{Lie}(\mathfrak{G})$ of $\mathfrak{G}$, respectively. 

See \cite[Remark 11]{M2} for a similar construction in an alternative situation. 
\end{rem}

\subsection{}\label{subsec:equivalence_theorem}
The following is our main result. 

\begin{theorem}\label{thm:equivalence}
We have a category equivalence $\mathsf{ASG}\approx \mathsf{HCP}$. In fact the functors
$\mathbf{P} : \mathsf{ASG}\to \mathsf{HCP}$ and $\G : \mathsf{HCP}\to \mathsf{ASG}$ are
quasi-inverse to each other. 
\end{theorem}

Since Proposition \ref{prop:PG=id} shows that $\mathbf{P}\circ \G$ is naturally isomorphic to
the identity functor $\mathrm{id}$, it remains to prove 
$\G \circ \mathbf{P}\simeq \mathrm{id}$. 

Let $\G \in \mathsf{ASG}$. Set
\[ 
\A := \O(\G),\quad \mathfrak{g}:= \mathrm{Lie}(\G),\quad \U:= \U(\mathfrak{g}),\quad G:=\G_{ev}. 
\]

\begin{lemma}\label{lem:U(g)toA*}
The natural embedding $\mathfrak{g} \subset \A^*$ uniquely extends to a superalgebra
map $\U \to \A^*$. The associated pairing $\langle \ , \ \rangle : \U \times \A \to \k$
is a Hopf pairing. 
\end{lemma}
\begin{proof}
The superalgebra map $\T(\mathfrak{g}) \to \A^*$ which extends $\mathfrak{g} \subset \A^*$
kills the first elements in \eqref{eq:homog_primitives}, by definition of the super-bracket.
For $v \in \mathfrak{g}_1$ it kills $2v^2-[v,v]$, whence it does $v^2-\frac{1}{2}[v,v]$ since
$\A^*$ is $2$-torsion free. This proves the first assertion.  

As for the second it is easy to see $\langle x, 1 \rangle = \varepsilon(x)$, $x \in \U$. 
It remains to prove
\[
\langle x, ab \rangle = \langle x_{(1)}, a \rangle \, \langle x_{(2)}, b \rangle,\quad
x\in \U, \ a, b \in \A.
\]
We may suppose that $x$ is of the form 
$x= u_1\dots u_r$, 
where $u_i$ are homogeneous elements in $\mathfrak{g}$.
Then the equation is proved by induction on the length $r$.  
\end{proof}

Recall $\A \in G\text{-}\mathsf{SMod}$,\ $\U \in \mathsf{SMod}\text{-}G$; see \eqref{eq:coad2} or 
\eqref{eq:gamma_a} as for $\A$, and see Lemma \ref{lem:G-stable} as for $\U$. 
Indeed, $\A$ and $\U$ are
Hopf-algebra objects in the respective categories. 

\begin{lemma}\label{lem:UA_paring(2)}
The Hopf pairing $\langle \ , \ \rangle : \U \times \A \to \k$ just obtained satisfies the same formula as 
\eqref{eq:UA_formula}. 
\end{lemma}
\begin{proof}
The $G$-module structure on $\mathfrak{g}$ is transposed from that on $\A^+/(\A^+)^2$. Therefore, the 
formula holds for every $x \in \mathfrak{g}$ and for any $a \in \A$. The desired formula follows
by induction, as in the last proof; see also the proof of Lemma \ref{lem:extended_G-action} (3). 
\end{proof}

Set 
\[
C :=\O(G), \quad J :=U(\mathfrak{g}_0).
\]
Note $ \mathbf{P}(\G) = (G, \mathfrak{g}) $. We aim to show that the affine Hopf superalgebra
$\A(G,\mathfrak{g})$, which is constructed from this last Harish-Chandra pair as in the previous subsection, 
is naturally isomorphic to the present $\A$. 
By using the Hopf pairing above and the notation \eqref{eq:quotient_map}, we define
\[
\eta' : \A \to \mathrm{Hom}_J(\U, C), \quad \eta'(a)(x) = \overline{a}_{(1)}\, \langle x, a_{(2)}\rangle,
\]
where $a \in \A$, $x \in \U$. Note that $\mathrm{Hom}_J(\U,C)$ has the Hopf superalgebra structure
which is transferred from $\A(G,\mathfrak{g})$ thorough $\eta$ (see \eqref{eq:A(G,g)Hom_J(U,C)}), and
which is presented by the formulas given in Proposition \ref{prop:transferred_structures}
with the obvious modification.  
(To answer a question by the referee we remark here that our $\eta'$ above 
is essentially the same, up to sign, as the existing ones such as
$\eta^*$ in \cite[Page 133, lines 2--3]{CCF}. 
The authors will discuss the difference of sign somewhere else.)

\begin{prop}\label{prop:GP=id}
$\eta'$ is an isomorphism of Hopf superalgebras. 
\end{prop}
\begin{proof}
Using Lemma \ref{lem:UA_paring(2)} one computes in
the same way as proving Proposition \ref{prop:transferred_structures} (2) so that
\[
\langle (\gamma, \delta),\ (\eta'(a_{(1)})\otimes \eta'(a_{(2)}))(x \otimes y) \rangle 
= \langle \gamma \delta,\ \eta'(a)(x^{\delta}\,  y) \rangle,
\]
where $a \in \A$, $\gamma, \delta \in G$, $x, y \in \U$. The right-hand side equals 
$\langle (\gamma, \delta),\ \Delta(\eta'(a)) (x \otimes y) \rangle$, by the formula
giving the coproduct on $\mathrm{Hom}_J(\U, C)$. Therefore, $\eta'$ preserves the coproduct. It
is easy to see that $\eta'$ preserves the remaining structure maps, and is hence a Hopf superalgebra map. 

Set $W :=W^{\A}$. Choose $\phi$ such as in \eqref{eq:phi}, and 
define $\varrho': \A \to \wedge(W)$ as $\varrho$ in
\eqref{eq:varrho}, with $\eta$ replaced by $\eta'$. 
Then as was seen for $\eta$ in the proof of Lemma \ref{lem:psi} (1), $\eta'$
is identified with 
\[
\psi' : \A \to C \otimes \wedge(W), \quad \psi'(a)=\overline{a}_{(1)}\otimes \varrho'(a_{(2)}).
\]
Since one sees that this $\psi'$ satisfies the assumption of Lemma \ref{lem:last_applied} below, 
the lemma proves that $\psi'$ and so $\eta'$ are isomorphisms.
\end{proof}

\begin{lemma}\label{lem:last_applied}
In general, let $\A$ be a split affine Hopf superalgebra, and set $C :=\overline{\A}$, $W:= W^{\A}$. 
Let $\psi : \A \to C \otimes \wedge(W)$ be a counit-preserving map of left $C$-comodule superalgebras.
Assume that the composite $(\varepsilon \otimes \varpi)\circ \psi : \A \to W$, where
$\varpi : \wedge(W) \to W$ denotes the canonical projection, coincides with the canonical projection
$\A \to \A_1/\A_0^+\A_1=W$. Then $\psi$ is necessarily an isomorphism.   
\end{lemma}
\begin{proof}
Let $\mathbf{B} := C\otimes \wedge(W)$. Set 
$\mathfrak{a}:= (\A_1)$ and $\mathfrak{b}:= (\mathbf{B}_1)\, (=C\otimes \wedge(W)^+)$
in $\A$ and in $\mathbf{B}$, respectively. Since $\psi(\mathfrak{a}^n) \subset \mathfrak{b}^n$ 
for every $n \ge 0$, 
there is induced
a counit-preserving, left $C$-comodule $\mathbb{N}$-graded algebra map
\[
\operatorname{gr}\psi : \operatorname{gr}\A = \bigoplus_{n=0}^{\infty} \mathfrak{a}^n/\mathfrak{a}^{n+1} \to 
\operatorname{gr}\mathbf{B} = \bigoplus_{n=0}^{\infty} \mathfrak{b}^n/\mathfrak{b}^{n+1}.
\]
One sees that $\operatorname{gr}\mathbf{B}= \mathbf{B}= C \otimes \wedge(W)$. 
Since $\A$ is split, we have as in \cite[Proposition 4.9 (2)]{M1},  
a canonical isomorphism  $\operatorname{gr}\A \simeq 
C \otimes \wedge(W)$, through which we will identify the two. 
Then $\operatorname{gr}\psi$ is a counit-preserving endomorphism of the left $C$-comodule
$\mathbb{N}$-graded algebra $C \otimes \wedge(W)$. Being a counit-preserving endomorphism of the
left $C$-comodule algebra $C$, 
$\operatorname{gr}\psi(0)$ is the identity on $C$. This together with the assumption above imply that
$\operatorname{gr}\psi(1)$ is the identity on $C\otimes W$. It follows that $\operatorname{gr}\psi$ is an
isomorphism. 
Since the affinity assumption implies $\operatorname{gr}\A(n)=0=\operatorname{gr}\mathbf{B}(n)$ for $n\gg 0$,
one sees that $\psi$ is an isomorphism.
\end{proof}

\begin{proof}[Proof of Theorem \ref{thm:equivalence}]
Since we see that $\eta$ and $\eta'$ are both natural, it follows that $\A(G, \mathfrak{g})$ and $\A$ 
are naturally isomorphic.
This proves $\G\circ \mathbf{P}\simeq \mathrm{id}$, as desired. 
\end{proof}

\begin{rem}\label{rem:compare_equivalences}
Suppose that $\k$ is a field of characteristic $\ne 2$. 
Identify $\mathsf{ASG}$ with $\mathsf{AHSA}$, 
through the obvious category anti-isomorphism.  
Identify our $\mathsf{HCP}$ defined by Definition \ref{def:HCP} with
that defined by \cite[Definition 7]{M2}, 
through the category anti-isomorphism given in 
Remark \ref{rem:compare_definitions}~(1). 
Then the category equivalences
$\mathbf{P}$ and $\G$ given by Theorem \ref{thm:equivalence} are easily identified with those
$A \mapsto (\overline{A}, W^A)$ and $(C, W) \mapsto A(C,W)$ 
given by \cite[Theorem 29]{M2}. 
\end{rem}

\section{$\G$-supermodules and $\mathrm{hy}(\G)$-supermodules}\label{sec:G-supermodules}
Throughout in this section we suppose that $\k$ is an integral domain.
Our assumption that $\k$ is $2$-torsion free is equivalent to that $2 \ne 0$ in $\k$. 

\subsection{}\label{hy(G)}
Let $\G \in \mathsf{ASG}$, and set $G:= \G_{ev}$. As before, we let  
$\A := \O(\G)$, whence $\overline{\A} = \O(G)$.  
We assume that $G$ is \emph{infinitesimally flat} \cite[Part~I, 7.4]{J}. 
This means that
\begin{itemize}
\item[(D1)] For every $n>0$,\ $\overline{\A}/(\overline{\A}^+)^n$ is $\k$-finite projective. 
\end{itemize}
By (C1), it follows that for every $n > 0$, $\A/(\A^+)^n$ is $\k$-finite projective. 

Recall that $\A^*$ is
the dual superalgebra of the super-coalgebra $\A$. 
We suppose $(\A/(\A^+)^n)^* \subset \A^*$ through the natural embedding, and set
\[
\mathrm{hy}(\G) := \bigcup_{n>0} (\A/(\A^+)^n)^*.  
\]
We call this the \emph{super-hyperalgebra} of $\G$. 
This is often denoted alternatively by $\operatorname{Dist}(\G)$, called 
the \emph{super-distribution algebra} of $\G$. 

It is easy to see that $\mathrm{hy}(\G)$ is a super-subalgebra of $\A^*$. 
By (D1), each $(\A/(\A^+)^n)^*$ is the dual coalgebra
of the algebra $\A/(\A^+)^n$. One sees that if $n < m$, then $(\A/(\A^+)^n)^*\subset (\A/(\A^+)^m)^*$ is a 
coalgebra embedding, so that all $(\A/(\A^+)^n)^*$, $n > 0$, form an inductive system of coalgebras. 

\begin{lemma}\label{lem:hy(G)}
Given the coalgebra structure of the inductive limit, 
the superalgebra $\mathrm{hy}(\G)$ forms a cocommutative Hopf superalgebra
such that the canonical pairing $\O(\G)^* \times \O(\G) \to \k$ restricts to a Hopf pairing
\begin{equation}\label{eq:hy(G)O(G)_pairing}
\langle \ , \ \rangle : \mathrm{hy}(\G)\times \O(\G) \to \k.
\end{equation}
\end{lemma}
\begin{proof}
Let $\H := \mathrm{hy}(\G)$.
Since each $(\A/(\A^+)^n)^*$ is cocommutative, so is $\H$.  
The dual $S^*$ of the antipode $S$ of $\A$ stabilizes $\H$. Denote $S^*|_{\H}$ by $S$. 
Then we see that the restricted pairing satisfies \eqref{eq:Hopf_pairing_conditions}, 
\eqref{eq:Hopf_pairing_antipode}. 
It follows that $\H$ satisfies the compatibility required to super-bialgebras (see \cite[Lemma~1]{M2}), 
and has $S= S^*|_{\H}$ as an antipode.  
\end{proof}

Let $\mathfrak{g}:=\operatorname{Lie}(\G)$. 
Note that the primitive elements in $\hy(\G)$ coincide precisely with $\mathfrak{g}$.
In addition, if $\k$ is a field of characteristic zero, then we have 
$\mathrm{hy}(\G) = \mathbf{U}(\mathfrak{g})$. 

The Hopf superalgebra quotient $\O(\G) \to \O(G)$ gives rise to a Hopf superalgebra
embedding of the hyperalgebra $\mathrm{hy}(G)$ of $G$ into $\mathrm{hy}(\G)$. Let 
$W :=W^{\A}\, (= \mathfrak{g}_1^*)$, 
and choose a counit-preserving isomorphism
\[ \psi : \O(\G) \overset{\simeq}{\longrightarrow} \O(G) \otimes \wedge(W) \] 
of left $\O(G)$-comodule superalgebras.

\begin{lemma}\label{lem:phi_from_psi}
There uniquely exists a unit-preserving isomorphism
\[
\phi : \mathrm{hy}(G) \otimes \wedge(\mathfrak{g}_1) \overset{\simeq}{\longrightarrow}
\mathrm{hy}(\G) 
\]
of left $\mathrm{hy}(G)$-module super-coalgebras such that 
\[
\langle \phi(z), a \rangle = \langle z, \psi(a) \rangle,  \quad a \in \O(\G),\
z \in \mathrm{hy}(G) \otimes \wedge(\mathfrak{g}_1),  
\]
where the right-hand side gives the tensor product of the canonical pairings 
\begin{equation}\label{eq:cano_pairings}
\mathrm{hy}(G) \times \O(G)\to \k,\quad \wedge(\mathfrak{g}_1) \times \wedge(W)\to \k. 
\end{equation}
\end{lemma}
\begin{proof}
We see that $\psi^*$ 
restricts to $\mathrm{hy}(G) \otimes \wedge(\mathfrak{g}_1) 
\overset{\simeq}{\longrightarrow} \mathrm{hy}(\G)$, and this isomorphism is such as mentioned above.  
\end{proof}

We will identify as
\begin{equation}\label{eq:identification}
\O(\G) = \O(G) \otimes \wedge(W),\quad  
\mathrm{hy}(G) \otimes \wedge(\mathfrak{g}_1) = \mathrm{hy}(\G)
\end{equation}
through $\psi$, $\phi$, respectively. 

Let $Q$  be the quotient field of $\k$, and let $G_Q$ denote the base change of $G$ to $Q$. 
In addition to (D1), we assume 
\begin{itemize}
\item[(D2)] $G_Q$ is connected, or in other words, $\mathscr{O}(G_Q)=\mathscr{O}(G)\otimes Q$ contains
no non-trivial idempotent. 
\end{itemize}

This assumption ensures the following. 

\begin{lemma}\label{lem:r-fold_injection}
For every $r >0$, the superalgebra map 
\[
\O(\G)^{\otimes r} \to (\mathrm{hy}(\G)^{\otimes r})^*
\]
which is associated with the $r$-fold tensor product of the Hopf pairing \eqref{eq:hy(G)O(G)_pairing} 
is injective. 
\end{lemma}
\begin{proof}
By Lemma \ref{lem:phi_from_psi} it suffices to prove that the algebra map 
$\O(G)^{\otimes r} \to (\mathrm{hy}(G)^{\otimes r})^*$ similarly given is injective. 
By \cite[Proposition~0.3.1(g)]{T1}, (D2) ensures that the $Q$-algebra map  
$\O(G_Q)^{\otimes r} \to (\mathrm{hy}(G_Q)^{\otimes r})^*$ for $G_Q$ is injective. 
Since $\hy(G_Q)=\hy(G)\otimes Q$, we have the canonical map $(\hy(G)^{\otimes r})^* \otimes Q
\to (\hy(G_Q)^{\otimes r})^*$. By (B3) we have $\O(G)^{\otimes r}  \subset \O(G)^{\otimes r} \otimes Q$. 
The desired injectivity follows from the commutative diagram
\[
\begin{xy}
(0,0)   *++{\O(G)^{\otimes r} \otimes Q}  ="1",
(38,0)  *++{(\hy(G)^{\otimes r})^* \otimes Q}    ="2",
(0,-15) *++{\O(G_Q)^{\otimes r}}          ="3",
(38,-15)*++{(\hy(G_Q)^{\otimes r})^*. }    ="4",
{"1" \SelectTips{cm}{} \ar @{->} "2"},
{"1" \SelectTips{cm}{} \ar @{->}^{\simeq} "3"},
{"2" \SelectTips{cm}{} \ar @{->} "4"},
{"3" \SelectTips{cm}{} \ar @{->} "4"}
\end{xy}
\]
\end{proof}

Let $M$ be a supermodule. Given a left $\G$-supermodule (resp., $G$-module) structure on $M$, 
one defines by the formula \eqref{eq:induced_supermodule}, using the Hopf pairing \eqref{eq:hy(G)O(G)_pairing} 
(resp., the first one of \eqref{eq:cano_pairings}), a left $\hy(\G)$-supermodule 
(resp.,  $\hy(G)$-module) structure on $M$.    
We see that in the super-situation, this indeed defines a map from
\begin{itemize}
\item
the set of all left $\G$-supermodule structures on $M$ 
\end{itemize}
to
\begin{itemize}
\item
the set of those locally finite, left
$\operatorname{hy}(\G)$-supermodule structures on $M$ whose restricted (necessarily, locally finite)
$\hy(G)$-module structures arise from left $G$-module structures. 
\end{itemize}

Note that the left and the right $\G$-supermodule structures 
(resp., locally finite $\operatorname{hy}(\G)$-supermodule structures with the property as above)
on $M$ are in one-to-one correspondence, since one can switch the sides through
the inverse on $\G$ (resp., the antipode on $\operatorname{hy}(\G)$). Therefore, we may 
replace $``$left" with $``$right" in the sets above, to prove the following proposition. Indeed, 
we do so, to make the argument fit in with our results so far obtained. 

\begin{prop}\label{prop:bijection_module_structures}
If $M$ is $\Bbbk$-projective, the map above is a bijection. 
\end{prop}
\begin{proof}
Since $M$ is $\Bbbk$-projective, the injection given by Lemma \ref{lem:r-fold_injection}, tensored with $M$,
remains injective. In addition the canonical map 
$(\mathrm{hy}(\G)^{\otimes r})^* \otimes M \to \mathrm{Hom}(\mathrm{hy}(\G)^{\otimes r}, M)$ 
is injective. Let 
\[ \mathbf{\mu}^{(r)} : \O(\G)^{\otimes r} \otimes M \to \mathrm{Hom}(\mathrm{hy}(\G)^{\otimes r}, M) \]
denote their composite, which is an injective super-linear map. 
We will use only $\mu^{(1)}$, $\mu^{(2)}$.

Suppose that we are given a structure from the second set; it is a \emph{right} 
$\mathrm{hy}(\G)$-supermodule structure, in particular. 
We claim that the super-linear map
\[ \rho : M \to \mathrm{Hom}(\mathrm{hy}(\G) , M), \quad \rho(m)(x) = m x \]
factorizes into $\mu^{(1)}$ and a uniquely determined map, $\rho' : M \to \mathscr{O}(\G)\otimes M$.
To show this we use the identification \eqref{eq:identification}.
Then, $\rho$ decomposes as
\[ M \overset{{\rho}_1}{\longrightarrow} \mathrm{Hom}(\mathrm{hy}(G), M) 
\overset{(\rho_2)_*}{\longrightarrow} 
\mathrm{Hom}(\mathrm{hy}(G), \mathrm{Hom}( \wedge (\mathfrak{g}_1) , M) ) , \]
where the first map is defined, just as $\rho$, by $\rho_1(m)(x) = m x$, and 
the second $(\rho_2)_*$ denotes 
$\mathrm{Hom}(\mathrm{id}, \rho_2)$ induced by  
the map $\rho_2 : M \to \mathrm{Hom}(\wedge(\mathfrak{g}_1), M) $ similarly defined. 
We have the injections
\begin{align*}
&\nu_1 : \O(G) \otimes M \to \mathrm{Hom}(\mathrm{hy}(G),M), \\ 
&\nu_2 : \O(G) \otimes \mathrm{Hom}(\wedge( \mathfrak{g}_1 ), M) \to \mathrm{Hom}(\mathrm{hy}(G), 
\mathrm{Hom}(\wedge ( \mathfrak{g}_1 ), M))
\end{align*}
which are defined in the same way as $\mu^{(1)}$. Indeed, $\nu_2$ is identified with $\mu^{(1)}$. 
The condition regarding the restricted $\mathrm{hy}(G)$-module structures means that 
$\rho_1$ factorizes into $\nu_1$ and a uniquely determined map, $\rho'': M \to \mathscr{O}(G) \otimes M$.
The composite $(\mathrm{id}\otimes \rho_2)\circ \rho''$ is identified with the desired map $\rho'$,  
as is seen from the commutative diagram
\[
\begin{xy}
(0,0)   *++{\O(G) \otimes M}  ="1",
(57,0)  *++{\O(G) \otimes \mathrm{Hom}(\wedge (\mathfrak{g}_1), M)}    ="2",
(0,-17) *++{\mathrm{Hom}(\mathrm{hy}(G), M)}          ="3",
(57,-17)*++{\mathrm{Hom}(\mathrm{hy}(G), \mathrm{Hom}(\wedge(\mathfrak{g}_1), M)).}    ="4",
{"1" \SelectTips{cm}{} \ar @{->}^{\mathrm{id}\otimes\rho_2\hspace{10mm}} "2"},
{"1" \SelectTips{cm}{} \ar @{->}^{\nu_1} "3"},
{"2" \SelectTips{cm}{} \ar @{->}^{\nu_2} "4"},
{"3" \SelectTips{cm}{} \ar @{->}^{(\rho_2)_* \hspace{10mm}} "4"}
\end{xy}
\]

By using $\mu^{(2)}$, we see that the associativity of the $\mathrm{hy}(\G)$-action on $M$
implies that $\rho' : M \to \mathscr{O}(\G)\otimes M$ is coassociative. 
Similarly, the unitality of the action implies that $\rho'$ 
is counital. Thus, $\rho'$ is a left $\mathscr{O}(\G)$-super-comodule structure on $M$. It is the unique
such structure that gives rise to the originally given structure,
as is easily seen. 
\end{proof}

\subsection{}\label{subsec:Chevalley_groups}
Let $G_{\mathbb{Z}}$ be a split reductive algebraic group over $\mathbb{Z}$; see \cite[p.153]{J}. 
By saying a reductive algebraic group we assume that it is connected and smooth. 
Choose a split maximal torus $T_{\mathbb{Z}}$. The pair $(G_{\mathbb{Z}}, T_{\mathbb{Z}})$
naturally corresponds to a root datum $(\mathsf{X}, \mathsf{R}, \mathsf{X}^\vee, \mathsf{R}^\vee)$.
In particular, $\mathsf{X}$ equals the character group $\mathsf{X}(T_{\mathbb{Z}})$ of $T_{\mathbb{Z}}$. 
It is known that $\O(G_{\mathbb{Z}})$ is $\mathbb{Z}$-free, and $G_{\mathbb{Z}}$ is infinitesimally flat.
Moreover, for any field $K$, 
the base change $(G_{\mathbb{Z}})_K$ is a split reductive (in particular, connected) algebraic group over $K$,
and $(T_{\mathbb{Z}})_K$ is its split maximal torus.
Conversely, every split reductive algebraic group over $K$ and its split maximal torus are 
obtained uniquely (up to isomorphism) in this manner.  

Recall that $\Bbbk$ is supposed to be an integral domain. Let 
\[ G = (G_{\mathbb{Z}})_{\Bbbk}, \quad T = (T_{\mathbb{Z}})_{\Bbbk} \]
be the base changes to $\Bbbk$. Note that $\O(G)$ is $\k$-free.
In addition, $G$ satisfies (D1) (with $\overline{A}$ supposed to be $\O(G)$) and (D2). 

We have the inclusion $\operatorname{hy}(G)\supset \operatorname{hy}(T)$ of hyperalgebras, which coincides
with the base changes of the hyperalgebras  
$\operatorname{hy}(G_{\mathbb{Z}})\supset \operatorname{hy}(T_{\mathbb{Z}})$
over $\mathbb{Z}$.  
Since $\Bbbk$ contains no non-trivial idempotent, the character group $\mathsf{X}(T)$ 
of $T$ remains to be $\mathsf{X}$. 

Let $M$ be a left or right $\operatorname{hy}(G)$-module. We say that $M$ is a 
$\operatorname{hy}(G)$-$T$-\emph{module} \cite[p.171]{J}, 
if the restricted $\operatorname{hy}(T)$-module structure on $M$ arises from some
$T$-module structure on it. This is equivalent to saying that $M$ is a direct sum
$M = \bigoplus_{\lambda\in \mathsf{X}}\, M_{\lambda}$ of $\Bbbk$-submodules 
$M_{\lambda}$, $\lambda \in \mathsf{X}$, so that  
\begin{equation*}
xm = \lambda(x) \, m,\quad 
x \in \operatorname{hy}(T),\ m \in M_{\lambda},\ \lambda \in \mathsf{X},  
\end{equation*} 
where we have supposed that $M$ is a \emph{left} $\operatorname{hy}(T)$-module. 
One sees that the $T$-module structure above is uniquely determined if $M$ is $\Bbbk$-torsion free.  
A $\operatorname{hy}(G)$-$T$-module is said to be \emph{locally finite} if it is locally finite as 
a $\operatorname{hy}(G)$-module.

Let $M$ be a $\k$-module. Given a left $G$-module structure on $M$, 
there arises, as before, a left $\hy(G)$-module structure on $M$; it is indeed a locally finite
$\operatorname{hy}(G)$-$T$-module structure, as is easily seen.  
Thus we have a map from 
\begin{itemize}
\item
the set of all left $G$-module structures on $M$ 
\end{itemize}
to
\begin{itemize}
\item
the set of all locally finite, left $\operatorname{hy}(G)$-$T$-module structure on $M$.
\end{itemize}

The structures in each set above are in one-to-one correspondence with
the opposite-sided structures, as before. The following is known. 

\begin{theorem}[\text{\cite[Part~II, 1.20, p.171]{J}}]\label{thm:bijection_non-super_case}
If $M$ is $\Bbbk$-projective, the map above is a bijection. 
\end{theorem}
\begin{rem}\label{rem:Kostant-Takeuchi}
Let $\Bbbk = \mathbb{Z}$, and suppose that $G_{\mathbb{Z}}$ is
semisimple, or equivalently $[\mathsf{X} : \mathbb{Z}\mathsf{R}] < \infty$; 
see \cite[Part~II, 1.6, p.158]{J}. 
Then it is known (see \cite{Kostant1, T2}) that 
\begin{equation}\label{eq:semisimple}
\mathscr{O}(G_{\mathbb{Z}})=\operatorname{hy}(G_{\mathbb{Z}})^{\circ}.
\end{equation}
It follows that every $\mathbb{Z}$-free, locally
finite $\operatorname{hy}(G_{\mathbb{Z}})$-module is necessarily a 
$\operatorname{hy}(G_{\mathbb{Z}})$-$T_{\mathbb{Z}}$-module. 

Given a Hopf algebra $H$ over $\mathbb{Z}$, 
we let $H^{\circ}$ denote, just when working over a field (see \cite[Section 6.0]{Sw}), 
the union of the $\mathbb{Z}$-submodules 
$(H/I)^*$ in $H^*$, where $I$ runs over the ideals of $H$ such that
$H/I$ is $\mathbb{Z}$-finite. Since the canonical map $(H/I)^*\otimes (H/I)^*\to (H/I \otimes H/I)^*$
is an isomorphism, each $(H/I)^*$ is a ($\mathbb{Z}$-finite free) coalgebra, whence $H^{\circ}$ is a coalgebra,
and is in fact a Hopf algebra. 
\end{rem} 

Keep $G$, $T$ as above. 
Let us consider objects $\G \in \mathsf{ASG}$ such that $\G_{ev} = G$. 

\begin{rem}\label{rem:supergroup_as_G}
(1) 
As will be seen Section \ref{subsec:super-Chevally_classical_type}, if $\k = \mathbb{Z}$, 
the \emph{Chevalley $\mathbb{Z}$-supergroups of classical type} which were constructed by
Fioresi and Gavarini \cite{FGmemo} and by Gavarini \cite{G21} 
(see also \cite{FGexp}) are examples of $\G$ as above. 
Therefore, their base changes are, as well. 

(2) 
Suppose that $\Bbbk$ is a field of characteristic $\ne 2$. Recall that every split reductive
algebraic group is of the form $G$ as above. Then it follows from 
Theorem \ref{thm:tensor_product_decomposition}
that the objects under consideration are precisely all 
algebraic supergroups $\G$ such that $\G_{ev}$ is a split reductive algebraic group. 
\end{rem}

Let $\G \in \mathsf{ASG}$ such that $\G_{ev} = G$. 

Let $M$ be a left or right $\operatorname{hy}(\G)$-supermodule.
We say that $M$ is a 
$\operatorname{hy}(\G)$-$T$-\emph{supermodule},
if the restricted $\operatorname{hy}(T)$-module structure on $M$ arises from some $T$-module structure on it;
this is equivalent to saying that $M$ is a 
$\operatorname{hy}(G)$-$T$-module, regarded as a $\operatorname{hy}(G)$-module by restriction. 
A $\operatorname{hy}(\G)$-$T$-{\emph{supermodule}} is said to be \emph{locally finite} 
if it is so as a $\operatorname{hy}(\G)$-supermodule, or equivalently, as a $\operatorname{hy}(G)$-module.

Let $M$ be a supermodule. Given a left $\G$-supermodule structure on $M$, there arises, as before, 
a left $\operatorname{hy}(\G)$-supermodule structure on $M$; it is indeed a locally finite  
$\operatorname{hy}(\G)$-$T$-supermodule structure, as is easily seen. 
Thus we have a map from
\begin{itemize}
\item
the set of all left $\G$-supermodule structures on $M$ 
\end{itemize}
to
\begin{itemize}
\item
the set of all locally finite, left
$\operatorname{hy}(\G)$-$T$-supermodule structures on $M$.
\end{itemize}

The structures in each set above are in one-to-one correspondence with
the opposite-sided structures, as before. 
Proposition \ref{prop:bijection_module_structures} and 
Theorem \ref{thm:bijection_non-super_case} prove the following. 

\begin{theorem}\label{thm:bijection_super_case}
If $M$ is $\Bbbk$-projective, the map above is a bijection. 
\end{theorem}

\begin{rem}\label{rem:bijection_super_case}
Let $\Bbbk = \mathbb{Z}$, and suppose that $G_{\mathbb{Z}}$ is semisimple.
Then by using the same argument as proving \cite[Proposition 31]{M2}, 
we see from \eqref{eq:semisimple} that
$\mathscr{O}(\G) =\operatorname{hy}(\G)^{\circ}$.  
It follows that every $\mathbb{Z}$-free, locally
finite $\operatorname{hy}(\G)$-supermodule is necessarily a 
$\operatorname{hy}(\G)$-$T_{\mathbb{Z}}$-supermodule.
\end{rem}

Theorem \ref{thm:bijection_super_case} can be reformulated as an isomorphism between 
the category of $\Bbbk$-projective, left $\G$-supermodules and the category of $\Bbbk$-projective,
locally finite left $\operatorname{hy}(\G)$-$T$-supermodules. When
$\Bbbk$ is a field of characteristic $\ne 2$, the result is formulated as follows, in view of 
Remark \ref{rem:supergroup_as_G} (2).  

\begin{corollary}\label{cor:category_isom}
Suppose that $\Bbbk$ is a field of characteristic $\ne 2$, and let $\G$ be an algebraic supergroup
over $\Bbbk$ such that $\G_{ev}$ is a split reductive algebraic group.
Choose a split maximal torus $T$ of $\G_{ev}$.
Then there is a natural isomorphism between the category of left $\G$-supermodules and the category of 
locally finite, left $\operatorname{hy}(\G)$-$T$-supermodules. 
\end{corollary}

This has been known only for some special algebraic supergroups with the property as above; 
see Brundan and Kleshchev \cite[Corollary 5.7]{BrKl}, Brundan and Kujawa \cite[Corollary 3.5]{BrKuj}, and
Shu and Wang \cite[Theorem 2.8]{ShuWang}.  

\section{Harish-Chandra pairs corresponding to Chevalley supergroups over $\mathbb{Z}$}\label{sec:Chevalley}

\subsection{}\label{subsec:Chevalley_intro}
Those finite-dimensional simple Lie superalgebras over the complex number field $\mathbb{C}$ which are not
purely even were classified by Kac \cite{Kac}. 
They are divided into classical type and Cartan type. A \emph{Chevalley} 
$\mathbb{C}$-\emph{supergroup of classical/Cartan type} 
is a connected algebraic supergroup $\G$ over $\mathbb{C}$ 
such that $\mathrm{Lie}(\G)$ is a simple Lie superalgebra of classical/Cartan type. 
As was mentioned in Remark \ref{rem:supergroup_as_G} (1), 
Fioresi and Gavarini \cite{FGmemo, G21} constructed natural $\mathbb{Z}$-forms of Chevalley
$\mathbb{C}$-supergroups of classical type. Gavarini \cite{G1} accomplished the same construction for Cartan type. 
The resulting $\mathbb{Z}$-forms are called 
\emph{Chevalley} $\mathbb{Z}$-\emph{supergroups of classical}/\emph{Cartan type}; 
they are indeed objects in our category $\mathsf{ASG}$ defined over $\mathbb{Z}$. 

Based on our Theorem \ref{thm:equivalence}, we will re-construct 
the Chevalley $\mathbb{Z}$-supergroups, by giving the corresponding 
Harish-Chandra pairs. Indeed, our construction depends on part of Fioresi and Gavarini's, but 
simplifies the rest; see Remarks \ref{rem:FGconstruction} and \ref{rem:Gconstruction}.  

\subsection{}\label{subsec:super-Chevally_classical_type}
Let $\mathfrak{g}$ be a finite-dimensional simple Lie superalgebra over $\mathbb{C}$
which is of classical type. Then $\mathfrak{g}_0$ is a reductive Lie algebra, and $\mathfrak{g}_1$, with
respect to the right adjoint $\mathfrak{g}_0$-action, decomposes as the direct sum of weight
spaces for a fixed Cartan subalgebra $\mathfrak{h} \subset \mathfrak{g}_0$. Let $\Delta_0$
(resp., $\Delta_1$) denote the set of the even (resp., odd) roots, that is, the weights with respect to
the adjoint $\mathfrak{h}$-action on $\mathfrak{g}_0$ (resp., on $\mathfrak{g}_1$). 

Let 
\begin{equation}\label{eq:root_datum}
(\mathsf{X}, \mathsf{R}, \mathsf{X}^\vee, \mathsf{R}^\vee),\quad 
G_{\mathbb{Z}}\supset T_{\mathbb{Z}}
\end{equation}
be a root datum and the corresponding split reductive algebraic $\mathbb{Z}$-group 
and split maximal torus.
Suppose that  
$\mathfrak{g}_0 \supset \mathfrak{h}$ coincide with the complexifications of
$\mathrm{Lie}(G_{\mathbb{Z}}) \supset \mathrm{Lie}(T_{\mathbb{Z}})$. Then one has 
\[
\mathsf{R}= \Delta_0, \quad \mathsf{X}^\vee \otimes_{\mathbb{Z}} \mathbb{C} = \mathfrak{h}, \quad 
\mathrm{hy}(G_{\mathbb{Z}})\otimes_{\mathbb{Z}} \mathbb{C} = U(\mathfrak{g}_0). 
\]
Recall that $\mathrm{hy}(G_{\mathbb{Z}})$ is called a \emph{Kostant form} of 
$U(\mathfrak{g}_0)$. 
We assume
\begin{equation}\label{eq:Delta1}
\Delta_1 \subset \mathsf{X}.
\end{equation}
  
\begin{theorem}[Fioresi, Gavarini]\label{thm:Chevalley_basis_classical}
There exists a $\mathbb{Z}$-lattice $V_{\mathbb{Z}}$ of $\mathfrak{g}_1$ such that
\begin{itemize}
\item[(i)] $\mathfrak{g}_{\mathbb{Z}}:= \mathrm{Lie}(G_{\mathbb{Z}}) \oplus V_{\mathbb{Z}}$
is a Lie-superalgebra $\mathbb{Z}$-form of $\mathfrak{g}$.
\item[(ii)] This Lie superalgebra $\mathfrak{g}_{\mathbb{Z}}$ over $\mathbb{Z}$ is admissible. 
\item[(iii)] $V_{\mathbb{Z}}$ is $\mathrm{hy}(G_{\mathbb{Z}})$-stable in the right 
$U(\mathfrak{g}_0)$-module $\mathfrak{g}_1$. 
\end{itemize}
\end{theorem}

Fioresi and Gavarini \cite{FGmemo} and Gavarini \cite{G21} 
introduced the notion of \emph{Chevalley bases}, gave an explicit example of such a basis for 
each $\mathfrak{g}$, and constructed 
from the basis a natural Hopf-superalgebra $\mathbb{Z}$-form, called a \emph{Kostant superalgebra},
of $\mathbf{U}(\mathfrak{g})$; the even basis elements
coincide with the classical Chevalley basis for $\mathfrak{g}_0$.  
They do not refer to root data.  But, once an explicit Chevalley basis is given
as in \cite{FGmemo, G21}, one can re-choose the basis so that it includes a $\mathbb{Z}$-free 
basis of $\mathsf{X}^{\vee}$, by replacing part of the original basis, $H_1,\dots, H_{\ell}$, 
with a desired $\mathbb{Z}$-free basis; this replacement is possible, since it effects only on 
the adjoint action on the basis elements $X_{\alpha}$, and the new basis elements still act via the roots $\alpha$. 
(The method of \cite[Remark 3.8]{FGmemo} attributed to the referee gives an alternative 
construction of the desired basis from the scratch.)
One sees that the odd elements in the Chevalley basis generate 
the desired $\mathbb{Z}$-lattice $V_{\mathbb{Z}}$ as above; see \cite[Sections 4.2, 6.1]{FGmemo} and 
\cite[Section 3.4]{G21}, to verify Condition (ii), in particular. 

Set $\mathfrak{g}_{\mathbb{Z}}:= \mathrm{Lie}(G_{\mathbb{Z}}) \oplus V_{\mathbb{Z}}$ in $\mathfrak{g}$,
as above. One sees from (iii) and \eqref{eq:Delta1} 
that $V_{\mathbb{Z}}$ is a right $\mathrm{hy}(G_{\mathbb{Z}})$-$T_{\mathbb{Z}}$-module,
whence it is a right $G_{\mathbb{Z}}$-module by Theorem \ref{thm:bijection_non-super_case}. 
The restricted super-bracket $[\ , \ ] : V_{\mathbb{Z}} \times V_{\mathbb{Z}} \to \mathrm{Lie}(G_{\mathbb{Z}})$,
being $\mathrm{hy}(G_{\mathbb{Z}})$-linear, is $G_{\mathbb{Z}}$-equivariant. This proves the following. 

\begin{prop}\label{prop:it's_HCP1}
$(G_{\mathbb{Z}}, \mathfrak{g}_{\mathbb{Z}})$ is a Harish-Chandra pair. 
\end{prop}

We let 
\[
\G_{\mathbb{Z}} = \G(G_{\mathbb{Z}}, \mathfrak{g}_{\mathbb{Z}})
\]
denote the algebraic $\mathbb{Z}$-supergroup in $\mathsf{ASG}$ 
which is associated with the Harish-Chandra pair just obtained. Since one sees that the category 
equivalences in Theorem \ref{thm:equivalence} are compatible with base extensions, 
it follows that $\G_{\mathbb{Z}}$ 
is a $\mathbb{Z}$-form of the algebraic $\mathbb{C}$-supergroup associated with the
Harish-Chandra pair $(G,\mathfrak{g})$, where $G$ denotes the base change of $G_{\mathbb{Z}}$ to $\mathbb{C}$. 
Recall from Section \ref{subsec:Chevalley_intro} the definition of Chevalley $\mathbb{C}$-supergroups
of classical type, and note that every such $\mathbb{C}$-supergroup is associated with 
some Harish-Chandra pair of the
last form. We have thus constructed 
a natural $\mathbb{Z}$-form of every Chevalley $\mathbb{C}$-supergroups of classical type. 

\begin{rem}\label{rem:FGconstruction}
(1)\
After constructing Kostant superalgebras, Fioresi and Gavarini's construction, which is parallel to
the classical construction of Chevalley $\mathbb{Z}$-groups, 
continues as follows; (a)~Choose a faithful rational representation
$\mathfrak{g} \to \mathfrak{gl}_{\mathbb{C}}(M)$ on a finite-dimensional super-vector space $M$ over $\mathbb{C}$, 
(b)~choose a $\mathbb{Z}$-lattice $M_{\mathbb{Z}}$ 
in $M$ which is stable under the action of the Kostant superalgebra,
(c)~construct a natural group-valued functor which is realized as subgroups of
$\mathbf{GL}_R(M_{\mathbb{Z}}\otimes_{\mathbb{Z}} R)$, 
where $R$ runs over the commutative superalgebras over $\mathbb{Z}$, and (d)~prove that the sheafification, 
say $\G^{FG}_{\mathbb{Z}}$, of the constructed
group-valued functor is representable, and has desired properties, which include the property that 
$\O(\G^{FG}_{\mathbb{Z}})$ is split; see \cite[Corollary 5.20]{FGmemo} and \cite[Corollary 4.22]{G21}
for the last property. 

Our method of construction dispenses with these procedures. 

(2)\
The algebraic group $(\G_{\mathbb{Z}}^{FG})_{ev}$ associated with Fioresi and Gavarini's $\G_{\mathbb{Z}}^{FG}$
is a split reductive algebraic $\mathbb{Z}$-group. 
As was noted in an earlier version of the present paper, it was not clear for the authors whether 
the split reductive algebraic $\mathbb{Z}$-groups which correspond to all \emph{possible} root data
(namely, all relevant root data satisfying \eqref{eq:Delta1})
can be realized as $(\G_{\mathbb{Z}}^{FG})_{ev}$; note that by definition, those algebraic $\mathbb{Z}$-groups 
are realized as our $(\mathbf{G}_{\mathbb{Z}})_{ev}=G_{\mathbb{\mathbb{Z}}}$. 
Later, Gavarini kindly showed to the first-named author that they are indeed realized; essentially
the same argument of his proof is contained in Erratum added to a new version of \cite{G1}. 
\end{rem}

\subsection{}\label{subsec:super-Chevally_Cartan_type}
Let $\mathfrak{g}$ be a finite-dimensional simple Lie superalgebra over $\mathbb{C}$
which is of Cartan type. Then $\mathfrak{g}_0$ is a direct sum 
$\mathfrak{g}_0^r\ltimes \mathfrak{g}_0^n$ of a reductive Lie algebra $\mathfrak{g}_0^r$ with a nilpotent
Lie algebra $\mathfrak{g}_0^n$. With respect to the right adjoint $\mathfrak{g}_0^r$-action,
$\mathfrak{g}_0^n$ and $\mathfrak{g}_1$ decompose as direct sums of
weight spaces for a fixed Cartan subalgebra
$\mathfrak{h}\subset \mathfrak{g}_0^r$; we let $\Delta_0^r$, $\Delta_0^n$ and $\Delta_1$ denote the
sets of the roots for $\mathfrak{g}_0^r$, $\mathfrak{g}_0^n$ and $\mathfrak{g}_1$, respectively. 
The nilpotent Lie algebra $\mathfrak{g}_0^n$ acts on $\mathfrak{g}_1$ nilpotently. 

This time we assume that the root datum and the corresponding algebraic $\mathbb{Z}$-groups 
given in \eqref{eq:root_datum}  
are as follows: $\mathfrak{g}_0^r \supset \mathfrak{h}$ coincide with the complexifications of
$\mathrm{Lie}(G_{\mathbb{Z}}) \supset \mathrm{Lie}(T_{\mathbb{Z}})$, and $\Delta_0^n \subset \mathsf{X}
\supset \Delta_1$.

\begin{theorem}[Gavarini]\label{thm:Chevalley_basis_Cartan}
There exist $\mathbb{Z}$-lattices $N_{\mathbb{Z}}$ and $V_{\mathbb{Z}}$ of $\mathfrak{g}_0^n$
and $\mathfrak{g}_1$, respectively, such that
\begin{itemize}
\item[(i)] $\mathfrak{g}_{\mathbb{Z}}:= \mathrm{Lie}(G_{\mathbb{Z}}) \oplus N_{\mathbb{Z}}\oplus V_{\mathbb{Z}}$
is a Lie-superalgebra $\mathbb{Z}$-form of $\mathfrak{g}$.
\item[(ii)] This Lie superalgebra $\mathfrak{g}_{\mathbb{Z}}$ over $\mathbb{Z}$ is admissible. 
\item[(iii)] $V_{\mathbb{Z}}$ is $\mathrm{hy}(G_{\mathbb{Z}})$-stable in the right 
$U(\mathfrak{g}_0^r)$-module $\mathfrak{g}_1$. 
\item[(iv)] $N_{\mathbb{Z}}$ contains a $\mathbb{Z}$-free basis $x_1,\dots, x_s$ such that
\begin{itemize}
\item[(iv-1)] the $\mathbb{Z}$-submodule $H_{\mathbb{Z}}$ of $U(\mathfrak{g}_0^n)$ which is (freely) generated by
\[
\frac{x_1^{n_1}}{n_1!}\dots \frac{x_s^{n_s}}{n_s!},\quad n_1\ge 0,\dots, n_s\ge 0
\]
is a $\mathbb{Z}$-subalgebra,
\item[(iv-2)] $V_{\mathbb{Z}}$ is $H_{\mathbb{Z}}$-stable in the right $U(\mathfrak{g}_0^n)$-module 
$\mathfrak{g}_1$, and
\item[(iv-3)] $H_{\mathbb{Z}}$ is $\hy(G_{\mathbb{Z}})$-stable in the right $U(\mathfrak{g}_0^r)$-module
$U(\mathfrak{g}_0^n)$. 
\end{itemize} 
\end{itemize}
\end{theorem}

Gavarini's construction in \cite{G1} is parallel to those in \cite{FGmemo, G21}. 
One sees that among Gavarini's Chevalley basis elements, the elements contained in $\mathfrak{g}_0^n$ and 
the odd elements generate the desired $\mathbb{Z}$-lattices $N_{\mathbb{Z}}$ and $V_{\mathbb{Z}}$, respectively;
the former are precisely the desired elements for (iv). 
See \cite[Section 3.1]{G1} for (ii), and see \cite[Section 3.3]{G1} for (iii), (iv).
Note that the $\mathbb{Z}$-algebra $H_{\mathbb{Z}}$ given in (iv-1)
is indeed a Hopf-algebra $\mathbb{Z}$-form of $U(\mathfrak{g}_0^n)$. 

Recall from \cite[IV, Sect. 2, 4.5]{DG}
there uniquely exists a unipotent algebraic group $F$ over $\mathbb{C}$ such that
$\mathrm{Lie}(F) = \mathfrak{g}_0^n$. The corresponding Hopf algebra $\O(F)$ is the
polynomial algebra $\mathbb{C}[t_1,\dots, t_s]$ such that 
\begin{equation}\label{eq:UO_pairing}
\langle \ , \ \rangle 
: U(\mathfrak{g}_0^n) \times \O(F) \to \mathbb{C}, \quad
\langle \frac{x_1^{n_1}}{n_1!}\dots \frac{x_s^{n_s}}{n_s!},\ t_1^{m_1}\dots t_s^{m_s} \rangle
= \delta_{n_1,m_1}\dots \delta_{n_s, m_s} 
\end{equation}
is a Hopf pairing. This induces a Hopf algebra isomorphism
\begin{equation}\label{eq:isom_from_O(F)}
\O(F) \overset{\simeq}{\longrightarrow} U(\mathfrak{g}_0^n)'.
\end{equation}
Here and in what follows, given a finitely generated Hopf algebra $B$ over a field or $\mathbb{Z}$,
we define 
\[ 
B' := \bigcup_{n>0} (B/(B^+)^n)^*,
\]
as in \cite[Section 9.2]{Mon}. 
This is a Hopf subalgebra of $B^{\circ}$. If $B$ is the commutative Hopf algebra corresponding
to an algebraic group, then $B'$ is the hyperalgebra of the algebraic group. 

\begin{lemma}\label{lem:unipotent_group}
$\mathbb{Z}[t_1,\dots, t_s]$ is a Hopf-algebra $\mathbb{Z}$-form of $\O(F) = \mathbb{C}[t_1,\dots, t_s]$. 
The Hopf pairing \eqref{eq:UO_pairing} over $\mathbb{C}$ restricts to a Hopf pairing 
$
\langle \ , \ \rangle : H_{\mathbb{Z}}
\times \mathbb{Z}[t_1,\dots, t_s] \to \mathbb{Z}
$
over $\mathbb{Z}$, and it induces an isomorphism 
\[ 
\mathbb{Z}[t_1,\dots, t_s] \overset{\simeq}{\longrightarrow} H_{\mathbb{Z}}'
\]
of $\mathbb{Z}$-Hopf algebras. 
\end{lemma}
\begin{proof} It is easy to see that the Hopf algebra isomorphism \eqref{eq:isom_from_O(F)}
restricts to a $\mathbb{Z}$-algebra map $\mathbb{Z}[t_1,\dots, t_n] \to H_{\mathbb{Z}}'$. 
We have the following commutative diagram which contains the isomorphism and the restricted algebra map. 
\[
\begin{xy}
(0,0)   *++{\mathbb{Z}[t_1,\dots, t_s]}  ="1",
(40,0)  *++{\O(F)=\mathbb{C}[t_1,\dots, t_s]} ="2",
(0,-15) *++{H_{\mathbb{Z}}'}          ="3",
(40,-15)*++{U(\mathfrak{g}_0^n)'}    ="4",
(0,-30) *++{H_{\mathbb{Z}}^*}          ="5",
(40,-30)*++{U(\mathfrak{g}_0^n)^*}    ="6",
{"1" \SelectTips{cm}{} \ar @{^(->} "2"},
{"1" \SelectTips{cm}{} \ar @{^(->} "3"},
{"2" \SelectTips{cm}{} \ar @{->}^{\simeq} "4"},
{"3" \SelectTips{cm}{} \ar @{^(->} "4"},
{"5" \SelectTips{cm}{} \ar @{^(->} "6"},
{"3" \SelectTips{cm}{} \ar @{^(->} "5"},
{"4" \SelectTips{cm}{} \ar @{^(->} "6"}
\end{xy}
\]
Since $H_{\mathbb{Z}}^* \simeq \mathbb{Z}[\hspace{-1.5pt}[t_1,\dots, t_n]\hspace{-1.5pt}]$,\ 
$U(\mathfrak{g}_0^n)^* 
\simeq \mathbb{C}[\hspace{-1.5pt}[t_1,\dots, t_n]\hspace{-1.5pt}]$, 
we see that the outer big square is a pull-back. The
lower square is a pull-back, too, as is easily seen. It follows that the upper square is a pull-back,
whence $\mathbb{Z}[t_1,\dots, t_n] \to H_{\mathbb{Z}}'$ is an isomorphism. This implies that 
$\mathbb{Z}[t_1,\dots, t_n]$ is a Hopf-algebra $\mathbb{Z}$-form of $\O(F)$. The rest is now easy to see. 
\end{proof}

Let $F_{\mathbb{Z}}$ denote the algebraic $\mathbb{Z}$-group corresponding to the $\mathbb{Z}$-Hopf algebra 
$\mathbb{Z}[t_1,\dots, t_s]$. Then
\[
\O(F_{\mathbb{Z}}) = \mathbb{Z}[t_1,\dots, t_s], 
\quad \mathrm{hy}(F_{\mathbb{Z}}) = H_{\mathbb{Z}},
\quad \mathrm{Lie}(F_{\mathbb{Z}})=N_{\mathbb{Z}}.
\]
Note from (i) of Theorem \ref{thm:Chevalley_basis_Cartan} that $N_{\mathbb{Z}}$ is a Lie-algebra
$\mathbb{Z}$-form of $\mathfrak{g}_0^n$. From the first two equalities above or from Gavarini's original
construction one sees that the construction of $H_{\mathbb{Z}}$
does not depend on the order of the basis elements.

Let $G \supset T$ denote the base changes of $G_{\mathbb{Z}}\supset T_{\mathbb{Z}}$ to $\mathbb{C}$. 
The right $U(\mathfrak{g}_0^r)$-module structure on $\mathfrak{g}_0^n$, which arises from the right
adjoint action, is indeed a 
$U(\mathfrak{g}_0^r)$-$T$-module structure. Hence it gives rise to a right $G$-module structure,
by which $\mathfrak{g}_0^n$ is a Lie-algebra object in the symmetric tensor category $\mathsf{Mod}$-$G$
of right $G$-modules. The structure uniquely extends to $U(\mathfrak{g}_0^n)$ so that 
$U(\mathfrak{g}_0^n)$ turns into a Hopf-algebra object in $\mathsf{Mod}$-$G$. One sees that the structure just
obtained is transposed through \eqref{eq:UO_pairing} to $\O(F)$, so that $\O(F)$ is a Hopf-algebra object in the
symmetric category $G$-$\mathsf{Mod}$ of left $G$-modules. Thus, $F$ turns into a right $G$-equivariant
algebraic group. The associated semi-direct product $G \ltimes F$ of algebraic groups has 
$\mathfrak{g}_0 = \mathfrak{g}_0^r \ltimes \mathfrak{g}_0^n$ as its Lie algebra, as is easily seen.
Note that $\mathfrak{g}_1$ is a right $U(\mathfrak{g}_0^r)$-$T$-module, and is such a right 
$U(\mathfrak{g}_0^n)$-module that is annihilated by $(U(\mathfrak{g}_0^n)^+)^m$ for some $m$. Then 
it follows that
$\mathfrak{g}_1$ turns into a right $G$-module and $F$-module. Moreover, it
is a right $G \ltimes F$-module, as is seen by using 
(1)~$\mathrm{Lie}(G \ltimes F)=\mathfrak{g}_0^r \ltimes \mathfrak{g}_0^n$, (2)~$G \ltimes F$ is connected, and  
(3)~$\mathfrak{g}_1$ is a right $U(\mathfrak{g}_0)$-module. 

What were constructed in the last paragraph are all defined over $\mathbb{Z}$, 
as is seen from the following Lemma. 
  
\begin{lemma}\label{lem:defined_over_Z}
Keep the notation as above.
\begin{itemize}
\item[(1)] The right $\O(G)$-comodule structure $\O(F) \to \O(F) \otimes_{\mathbb{C}} \O(G)$ on $\O(F)$ 
restricts to $\O(F_{\mathbb{Z}}) \to \O(F_{\mathbb{Z}})\otimes_{\mathbb{Z}} \O(G_{\mathbb{Z}})$, by which
$F_{\mathbb{Z}}$ turns into a right $G_{\mathbb{Z}}$-equivariant algebraic group. Therefore, we have the
associated semi-direct product $G_{\mathbb{Z}}\ltimes F_{\mathbb{Z}}$ of algebraic groups. 
\item[(2)] $V_{\mathbb{Z}}$ is naturally a right $G_{\mathbb{Z}}\ltimes F_{\mathbb{Z}}$-module.
\end{itemize}
\end{lemma}
\begin{proof}
(1)\
One sees that the right $\hy(G_{\mathbb{Z}})$-module structure on $H_{\mathbb{Z}}$ which is given by (iv-3) 
of Theorem \ref{thm:Chevalley_basis_Cartan}
is indeed a $\hy(G_{\mathbb{Z}})$-$T_{\mathbb{Z}}$-module structure. Hence
it gives rise to a right $G_{\mathbb{Z}}$-module structure on $H_{\mathbb{Z}}$, by which $H_{\mathbb{Z}}$
turns into a Hopf-algebra object in $\mathsf{Mod}$-$G_{\mathbb{Z}}$. 
Since the isomorphism given in 
Lemma \ref{lem:unipotent_group} is compatible with base extension, it follows that the last structure 
is transposed to a left $G_{\mathbb{Z}}$-module structure on $\O(F_{\mathbb{Z}})$, so that
$\O(F_{\mathbb{Z}})$ is a Hopf-algebra object in $G_{\mathbb{Z}}$-$\mathsf{Mod}$. By construction
the corresponding right $\O(G_{\mathbb{Z}})$-comodule structure on $\O(F_{\mathbb{Z}})$
is the restriction of the right $\O(G)$-comodule structure on $\O(F)$. This proves 
the first assertion. The rest is easy to see.

(2)\
Just as for $H_{\mathbb{Z}}$, we see from (iii) of the theorem that
$V_{\mathbb{Z}}$ is a right $\hy(G_{\mathbb{Z}})$-$T_{\mathbb{Z}}$-module, whence it is a right 
$G_{\mathbb{Z}}$-module. We see from (iv-2) that $V_{\mathbb{Z}}$ is a right $H_{\mathbb{Z}}$-module,
and it is indeed a right $H_{\mathbb{Z}}/(H_{\mathbb{Z}}^+)^m$-module for the same $m$ as before. It follows by 
Lemma \ref{lem:unipotent_group} that $V_{\mathbb{Z}}$ is a right $F_{\mathbb{Z}}$-module.

It remains to prove that
\[
(vf)g = (vg)f^g,\quad 
v \in V_{\mathbb{Z}},\ f \in F_{\mathbb{Z}},\ g \in G_{\mathbb{Z}}.
\] 
Let $R$ be a commutative
ring. The equality in $R\otimes_{\mathbb{Z}}\mathbb{C}$-points follows from the analogous equality for
$\mathfrak{g}_1$, since 
$V_{\mathbb{Z}}\otimes_{\mathbb{Z}} R \otimes_{\mathbb{Z}} \mathbb{C}= 
\mathfrak{g}_1\otimes_{\mathbb{C}} (R \otimes_{\mathbb{Z}} \mathbb{C})$. 
To prove the equality in $R$-points, we may suppose $R = \O(F_{\mathbb{Z}})\otimes_{\mathbb{Z}}\O(G_{\mathbb{Z}})$,
and so that $R$ is $\mathbb{Z}$-flat. In this case the equality follows from the previous result 
since we then have
$V_{\mathbb{Z}}\otimes_{\mathbb{Z}} R \subset V_{\mathbb{Z}}\otimes_{\mathbb{Z}} R \otimes_{\mathbb{Z}} \mathbb{C}$. 
\end{proof}

Recall that $\mathfrak{g}_{\mathbb{Z}}$ is a Lie-superalgebra $\mathbb{Z}$-form as given in
(i) of Theorem \ref{thm:Chevalley_basis_Cartan}. Its odd component $V_{\mathbb{Z}}$ is a 
right $G_{\mathbb{Z}}\ltimes F_{\mathbb{Z}}$-module by Lemma \ref{lem:defined_over_Z}. 

\begin{prop}\label{prop:it's_HCP2}
$(G_{\mathbb{Z}}\ltimes F_{\mathbb{Z}}, \mathfrak{g}_{\mathbb{Z}})$ is a Harish-Chandra pair. 
\end{prop}
\begin{proof}
As is easily seen, $\mathrm{Lie}(G_{\mathbb{Z}}\ltimes F_{\mathbb{Z}})$ coincides with the
even component $\mathrm{Lie}(G_{\mathbb{Z}})\ltimes N_{\mathbb{Z}}$ of $\mathfrak{g}_{\mathbb{Z}}$.
The restricted super-bracket $[ \ , \ ] : V_{\mathbb{Z}}\times V_{\mathbb{Z}}\to 
\mathrm{Lie}(G_{\mathbb{Z}}\ltimes F_{\mathbb{Z}})$, being $\hy(G_{\mathbb{Z}})$- and $H_{\mathbb{Z}}$-linear, 
is $G_{\mathbb{Z}}$- and $F_{\mathbb{Z}}$-equivariant. It is necessarily 
$G_{\mathbb{Z}}\ltimes F_{\mathbb{Z}}$-equivariant.  
\end{proof}

We have thus the algebraic $\mathbb{Z}$-supergroup 
$\G(G_{\mathbb{Z}}\ltimes F_{\mathbb{Z}}, \mathfrak{g}_{\mathbb{Z}})$ in $\mathsf{ASG}$
which is associated with the Harish-Chandra pair just obtained. It is a $\mathbb{Z}$-form of the algebraic 
$\mathbb{C}$-supergroup which is associated with the Harish-Chandra pair $(G \ltimes F, \mathfrak{g})$. Since
every Chevalley $\mathbb{C}$-supergroup of Cartan type (see Section \ref{subsec:Chevalley_intro})
is associated with some Harish-Chandra pair of the last form, we have constructed 
a natural $\mathbb{Z}$-form of every such $\mathbb{C}$-supergroup. 

\begin{rem}\label{rem:Gconstruction}
Just as in the classical-type case (see Remark \ref{rem:FGconstruction} (1)), 
Gavarini's construction requires 
faithful representations of $\mathfrak{g}$, which, however, must satisfy more involved conditions
as given in \cite[Definition 3.14]{G1}; Proposition 3.16 of \cite{G1} proves that part of the conditions
is satisfied if the representation is completely reducible. 
The required representations look thus rather restrictive. On the other hand, 
Theorem 4.42 of \cite{G1} implies that the required representations are many enough to ensure
that our $\mathbb{Z}$-forms all are realized by Gavarini's construction. 
But the proof of the theorem is wrong, as was pointed out in an earlier version of this paper. 
After the publication of \cite{G1}, a corrected proof 
of the theorem, which uses the category equivalence \cite[Theorem 4.3.14]{G2}\, 
(= Theorem \ref{thma:Gavarini's_equivalence} below), 
was given in Erratum added to a new version of \cite{G1}. As far as the authors see, the proof is correct
if the same argument as proving our Lemma \ref{lem:defined_over_Z} is added. 
\end{rem}

\appendix
\section{Generalization using $2$-operations}\label{seca:generalization}

In this appendix we work over an arbitrary non-zero commutative ring $\Bbbk$. As was announced
at the last paragraph of the Introduction we will refine Gavarini's category
equivalence; see Theorem \ref{thma:Gavarini's_equivalence}. 

\subsection{}\label{subseca:preliminaries_on_2-operations}
Let $\mathfrak{g}$ be a \emph{Lie superalgebra}; see Section \ref{subsec:definition_of_admissible_LSA}. 

\begin{definition}[\text{\cite[Definition 2.2.1]{G2}}]\label{defa:2-operation}
A $2$-\emph{operation} on $\mathfrak{g}$ is a map  
$(\ )^{\langle 2 \rangle} : \mathfrak{g}_1\to \mathfrak{g}_0$ such that
\begin{itemize}
\item[(i)] $(cv)^{\langle 2 \rangle} = c^2v^{\langle 2 \rangle}$,
\item[(ii)] $(v+w)^{\langle 2 \rangle} = v^{\langle 2 \rangle} +[v,w]+ w^{\langle 2 \rangle}$ and
\item[(iii)] $[v^{\langle 2 \rangle}, z] = [v,[v,z]]$,
\end{itemize}
where $c \in \k$, $v, w \in \mathfrak{g}_1$, $z\in \mathfrak{g}$. 
\end{definition}

This is related with the admissibility defined by Definition \ref{def:admissible_Lie} as follows. 

\begin{lemma}\label{lema:2-operation}
Assume that $\Bbbk$ is $2$-torsion free. If $\mathfrak{g}$ is admissible, then
\[ v^{ \langle 2 \rangle} := \frac{1}{2}[v,v],\ v \in \mathfrak{g}_1 \]
gives the unique $2$-operation on $\mathfrak{g}$, and this is indeed the unique 
map $\mathfrak{g}_1\to \mathfrak{g}_0$ that satisfies (i), (ii) above. 
\end{lemma}
\begin{proof}
The left and the right-hand sides of (i)--(iii) coincide since their doubles are seen to coincide.  
The uniqueness follows, since we see from (i), (ii) that
$4v^{\langle 2 \rangle}=(2v)^{\langle 2 \rangle}= 2v^{\langle 2 \rangle}+ [v,v]$, 
and so $2v^{\langle 2 \rangle}=[v,v]$. 
\end{proof}

If $\Bbbk$ is $2$-torsion free, an admissible Lie superalgebra is thus the same as a Lie superalgebra
$\mathfrak{g}$ given a (unique) $2$-operation, such that $\mathfrak{g}_0$ is $\Bbbk$-flat and $\mathfrak{g}_1$
is $\Bbbk$-free.  

Let us return to the situation that $\Bbbk$ is arbitrary. Let $\mathfrak{g}$ be a Lie superalgebra
given a $2$-operation. 
One directly verifies the following.  

\begin{prop}\label{propa:base_extension_of_2-operation}
Suppose that the odd component $\mathfrak{g}_1$ is $\Bbbk$-free, and choose a totally ordered
basis $X$ arbitrarily. Given a commutative algebra $R$, define a map
\[
(\ )_R^{\langle 2 \rangle} : \mathfrak{g}_1 \otimes R \to \mathfrak{g}_0 \otimes R
\]
by 
\[
\big(\sum_{i=1}^n x_i \otimes c_i\big)_R^{\langle 2 \rangle}:= 
\sum_{i=1}^nx_i^{\langle 2 \rangle}\otimes c_i^2+ \sum_{i<j}[x_i,x_j]\otimes c_ic_j,
\]
where $x_1 <\dots < x_n$ in $X$, and $c_i \in R$. This definition is independent of choice
of ordered bases, and the map gives a $2$-operation on the $R$-Lie superalgebra $\mathfrak{g} \otimes R$. 
For arbitrary elements $v_i \in \mathfrak{g}_1$, $c_i \in R$, $1\le i \le m$, we have
\[
\big(\sum_{i=1}^m v_i \otimes c_i\big)_R^{\langle 2 \rangle}= 
\sum_{i=1}^m v_i^{\langle 2 \rangle}\otimes c_i^2+ \sum_{i<j}[v_i,v_j]\otimes c_i c_j.
\]
\end{prop}

In this appendix we let $\mathbf{U}(\mathfrak{g})$ denote the cocommutative
Hopf superalgebra which is defined as in \cite[Section 4.3.4]{G2}. This is the quotient 
Hopf superalgebra of the tensor algebra $\mathbf{T}(\mathfrak{g})$ divided by the super-ideal
generated by the homogeneous primitives
\begin{equation*}
zw -(-1)^{|z||w|}wz-[z,w],\quad v^2 - v^{\langle 2 \rangle}, 
\end{equation*}
where $z$ and $w$ are homogeneous elements in $\mathfrak{g}$, and $v \in \mathfrak{g}_1$.
The only difference from the definition given in Section \ref{subsec:co-splitting_result} is 
that the second generators $v^2 - \frac{1}{2}[v,v]$ 
in \eqref{eq:homog_primitives} are here replaced (indeed, generalized) by $v^2-v^{\langle 2 \rangle}$.

\begin{lemma}\label{lema:monomial_basis}
Suppose that the homogeneous components $\mathfrak{g}_0$ and $\mathfrak{g}_1$ are both $\Bbbk$-free,
and choose their totally ordered bases $X_0$ and $X_1$. Then $\mathbf{U}(\mathfrak{g})$
has the following monomials as a $\Bbbk$-free basis,
\[
a_1^{r_1}\dots a_m^{r_m}x_1\dots x_n,
\]
where $a_1< \dots <a_m$ in $X_0$, $r_i > 0$, $m \ge 0$, and $x_1<\dots <x_n$ in $X_1$, $n \ge 0$. 
\end{lemma}
\begin{proof}
To prove Proposition \ref{prop:co-split} we used the Diamond Lemma \cite[Proposition 7.1]{B} for $R$-rings.
But we use here the Diamond Lemma \cite[Theorem 1.2]{B} for $\Bbbk$-algebras. We suppose that $X_0 \cup X_1$ is the
set of generators, and extend
the total orders on $X_i$, $i=0,1$, to the set so that $a < x$ whenever $a \in X_0$,
$x \in X_1$. The reduction system consists of the obvious reductions arising from the super-bracket, and
\[ x^2 \to x^{\langle 2 \rangle},\quad x \in X_1, \]
where the last $x^{\langle 2 \rangle}$ is supposed to be
presented as a linear combination of elements in $X_0$. It is essential to prove that the overlap
ambiguities which may occur when we reduce the words
\begin{itemize}
\item $xxa$,\quad $x\in X_1$, $a \in X_0$,
\item $xyz$, \quad $x= y\ge z$ or $x\ge y= z$ in $X_1$ 
\end{itemize}
are resolvable. 
This is easily proved (indeed, more easily than was in the proof of Proposition \ref{prop:co-split}),
by using Condition (iii) in Definition \ref{defa:2-operation}. For example, the word $xxa$ is reduced
on the one hand as
\[ xxa \to x[x,a]+ xax \to x[x,a] + [x,a]x+ ax^{\langle 2 \rangle}\to[x,[x,a]]+ax^{\langle 2 \rangle}, \]
and on the other hand as
\[ xxa \to x^{\langle 2 \rangle}a . \]
The two results coincide by (iii). 
\end{proof}

\begin{rem}
To use Condition (iii) as above, we cannot treat $\mathbf{U}(\mathfrak{g})$ as a $J=U(\mathfrak{g}_0)$-ring 
as in the the proof of Proposition \ref{prop:co-split}. 
Indeed, to reduce the word $xxa$ with $a \in J$ in the proof, we are not allowed to present 
$a$ as (a linear combination of) $bc$ with $b \in \mathfrak{g}_0$, $c \in J$, and to reduce as
\[
xxa \to xxbc \to x[x,b]c + xbxc,  
\]
because by the first step, the lengths of words increase, 
$\mathrm{length}(xx*) < \mathrm{length}(xx**)$;
see the proof of \cite[Lemma 11]{M2}.  
\end{rem}

\begin{corollary}[\text{cf.~\cite[(4.7)]{G2}}]\label{cora:co-split}
If $\mathfrak{g}_0$ is $\Bbbk$-finite projective and $\mathfrak{g}_1$ is $\Bbbk$-free, then the
same result as Corollary \ref{cor:co-split} holds, that is, there exists a unit-preserving,
left $U(\mathfrak{g}_0)$-module super-coalgebra isomorphism
$U(\mathfrak{g}_0)\otimes \wedge(\mathfrak{g}_1) \overset{\simeq}{\longrightarrow}
\U(\mathfrak{g})$.
\end{corollary}
\begin{proof}
Choose a totally ordered basis $X$ of $\mathfrak{g}_1$, and define a left $U(\mathfrak{g}_0)$-module
(super-coalgebra) map $\phi : U(\mathfrak{g}_0)\otimes \wedge(\mathfrak{g}_1)\to \U(\mathfrak{g})$ by
\[ \phi(1 \otimes (x_1\wedge\dots \wedge x_n)) = x_1\dots x_n, \]
where $x_1\ <\dots < x_n$ in $X$, $n \ge 0$. To prove that this is bijective, it suffices to prove the
localization $\phi_{\mathfrak{m}}$ at each maximal ideal $\mathfrak{m}$ of $\Bbbk$ is bijective. Note that
$\mathfrak{g}_{\mathfrak{m}}$ is a
$\Bbbk_{\mathfrak{m}}$-Lie superalgebra given a $2$-operation by 
Proposition \ref{propa:base_extension_of_2-operation}, and
\[
U(\mathfrak{g}_0)_{\mathfrak{m}} = U((\mathfrak{g}_0)_{\mathfrak{m}}),\quad 
(\wedge(\mathfrak{g}_1))_{\mathfrak{m}} = \wedge((\mathfrak{g}_1)_{\mathfrak{m}}),\quad 
\U(\mathfrak{g})_{\mathfrak{m}} = \U(\mathfrak{g}_{\mathfrak{m}}).
\] 
Since $(\mathfrak{g}_0)_{\mathfrak{m}}$
is $\Bbbk_{\mathfrak{m}}$-free under the assumption above, Lemma \ref{lema:monomial_basis} 
shows that $\phi_{\mathfrak{m}}$ is bijective. 
\end{proof}

Let $\G$ be an \emph{affine supergroup}; see Section \ref{subsec:ASG_basics}. Recall from 
Section \ref{subsec:Lie(G)}
\[ 
\mathrm{Lie}(\G) := (\O(\G)^+/(\O(\G)^+)^2)^* .
\]
Note that the proof of Proposition \ref{prop:Lie(G)} does not use the assumption that 
$\Bbbk$ is $2$-torsion free. From the proposition and the proof one sees the following.

\begin{prop}\label{propa:Lie(G)}
Let $\mathfrak{g}:= \mathrm{Lie}(\G)$. 
\begin{itemize}
\item[(1)] $\mathfrak{g}$ is naturally a Lie superalgebra. 
\item[(2)] Given $v \in \mathfrak{g}_1$, the square $v^2$ in $\O(\G)^*$ is contained in $\mathfrak{g}_0$.
Moreover, the square map $(\ ) ^2 : \mathfrak{g}_1 \to \mathfrak{g}_0$ gives a $2$-operation on $\mathfrak{g}$.
\end{itemize}
\end{prop}

We will suppose that $\mathrm{Lie}(\G)$ is given this specific $2$-operation. 

\subsection{}\label{subseca:Gavarini's_categories}
Recall from \cite[Definitions 3.2.6 and 4.1.2]{G2} the following definitions of two categories,
$(\mathrm{gss}\text{-}\mathrm{fsgroups})_{\k}$, $(\mathrm{sHCP})_{\k}$. 

Let $(\mathrm{gss}\text{-}\mathrm{fsgroups})_{\k}$ denote 
the category of the affine supergroups $\G$ such that
when we set $\A := \mathscr{O}(\G)$, 
\begin{itemize}
\item[(E1)] $\A$ is \emph{split} (Definition \ref{def:split}), 
\item[(E2)] $\overline{\A}/(\overline{\A}^+)^2$ is $\Bbbk$-finite projective, and
\item[(E3)] $W^{\A}= \A_1/\A_0^+ \A_1$ is $\Bbbk$-finite (free). 
\end{itemize}
The morphisms in $(\mathrm{gss}\text{-}\mathrm{fsgroups})_{\k}$ are the natural transformations of 
group-valued functors. 

Let $(G, \mathfrak{g})$ be a pair of an affine group $G$ and a Lie superalgebra $\mathfrak{g}$ 
given a $2$-operation, such that $\mathfrak{g}_1$ is $\Bbbk$-finite free and is given a right $G$-module
structure. Suppose that this pair satisfies 
\begin{itemize}
\item[(F1)] $\mathfrak{g}_0 = \mathrm{Lie}(G)$,
\item[(F2)] $\mathscr{O}(G)/(\mathscr{O}(G)^+)^2$ is $\Bbbk$-finite projective, so that   
$\mathfrak{g}_0 = \mathrm{Lie}(G)$ is necessarily 
$\Bbbk$-finite projective, and it is naturally a right $G$-module (recall from Section \ref{subsec:HCP}
that the corresponding left $\O(G)$-comodule structure on $\mathrm{Lie}(G)$ is transposed from
the right co-adjoint $\O(G)$-coaction on 
$\O(G)^+/(\O(G)^+)^2$),  
\item[(F3)] the right $U(\mathfrak{g}_0)$-module structure on $\mathfrak{g}_1$ induced from the
given right $G$-module structure coincides with the right adjoint $\mathfrak{g}_0$-action on $\mathfrak{g}_1$,
\item[(F4)] the restricted super-bracket $[\ , \ ] : \mathfrak{g}_1 \otimes \mathfrak{g}_1\to \mathfrak{g}_0$ 
is $G$-equivariant, and
\item[(F5)] the diagram 
\[
\begin{xy}
(0,0)   *++{\mathfrak{g}_1}  ="1",
(35,0)  *++{\mathfrak{g}_0}    ="2",
(0,-15) *++{\O(G)\otimes \mathfrak{g}_1}          ="3",
(35,-15)*++{\O(G)\otimes \mathfrak{g}_0}    ="4",
{"1" \SelectTips{cm}{} \ar @{->}^{(\ )^{\langle 2 \rangle}} "2"},
{"1" \SelectTips{cm}{} \ar @{->} "3"},
{"2" \SelectTips{cm}{} \ar @{->} "4"},
{"3" \SelectTips{cm}{} \ar @{->}^{(\ )_{\mathscr{O}(G)}^{\langle 2 \rangle}} "4"}
\end{xy}
\] 
commutes, where the vertical arrows are the left $\O(G)$-comodule structures. 
\end{itemize}
One sees that under (F4), Condition (F5) is equivalent to 
\[
(v_R^{\langle 2 \rangle})^{\gamma} =(v^{\gamma})_R^{\langle 2 \rangle}, \quad 
v \in \mathfrak{g}_1\otimes R,\ \gamma \in G(R),
\]
where $R$ is an arbitrary commutative algebra.

Let $(\mathrm{sHCP})_{\k}$ denote the category of all those pairs $(G, \mathfrak{g})$ 
which satisfy Conditions (F1)--(F5) above. 
A morphism $(G, \mathfrak{g}) \to (G', \mathfrak{g}')$ in $(\mathrm{sHCP})_{\k}$
is a pair $(\alpha, \beta)$ of a morphism  
$\alpha: G \to G'$ of affine groups and a Lie superalgebra map 
$\beta = \beta_0 \oplus \beta_1 : \mathfrak{g} \to \mathfrak{g}'$, which satisfies Conditions (i), (ii) in 
Definition \ref{def:HCP}, and
\begin{itemize}
\item[(iii)] $\beta_0(v^{\langle 2 \rangle}) = \beta_1(v)^{\langle 2 \rangle}$,\quad $v \in \mathfrak{g}_1$.
\end{itemize} 

\begin{rem}\label{rema:compare_definitions}
One sees from Lemma \ref{lema:2-operation} that 
if $\Bbbk$ is $2$-torsion free, then our $\mathsf{HCP}$ and $\mathsf{ASG}$ 
(see Definition \ref{def:HCP} and Section \ref{subsec:P}), roughly speaking, coincide
with $(\mathrm{sHCP})_{\k}$ and $(\mathrm{gss}\text{-}\mathrm{fsgroups})_{\k}$, respectively.
To be precise, ours are more restrictive in that for objects $(G, \mathfrak{g})\in 
\mathsf{HCP}$, $\G \in \mathsf{ASG}$, the commutative Hopf algebras $\O(G)$ and $\O(\G_{ev})$
are assumed to be affine and $\k$-flat. 

We may remove the affinity assumption so long as 
(B1) and (C3) are assumed. 
But the assumption seems natural, since if $\Bbbk$ is a field of characteristic $\ne 2$, it ensures 
that (B1) and (C3) are satisfied, so that our Theorem \ref{thm:equivalence}
then coincides with the known category equivalence
between all algebraic supergroups and the Harish-Chandra pairs; see Remark \ref{rem:compare_equivalences}.

Note from \eqref{eq:verify(F5)} that under the $\Bbbk$-flatness assumption above, 
$\O(G) \otimes \mathfrak{g}_1$ is $2$-torsion free, 
and Condition (F5) for $v^{\langle 2 \rangle} = \frac{1}{2}[v,v]$ is necessarily satisfied.
Recall that the condition is not contained in the axioms for objects in $\mathsf{HCP}$.   
\end{rem}

\subsection{}\label{subseca:Gavarini's_equivalence}
Our category equivalences between $(\mathrm{gss}\text{-}\mathrm{fsgroups})_{\k}$ and $(\mathrm{sHCP})_{\k}$
will be presented differently from 
Gavarini's $\Phi_g$, $\Psi_g$; see Remark \ref{rema:compare_functors}. 
So, we will use different symbols, $\mathbf{P}'$, $\G'$, to denote them.   

Let us construct a functor $\mathbf{P}' : (\mathrm{gss}\text{-}\mathrm{fsgroups})_{\k} \to (\mathrm{sHCP})_{\k}$.
Given $\G \in (\mathrm{gss}\text{-}\mathrm{fsgroups})_{\k}$, 
set $G :=\G_{ev}$, $\mathfrak{g}:= \mathrm{Lie}(\G)$. Recall from Proposition 
\ref{propa:Lie(G)} and the following remark that
$\mathfrak{g}$ is a Lie superalgebra given the square map as a $2$-operation. 
As in Lemma \ref{lem:Lie(G)g_0} 
one has $\mathfrak{g}_0 \simeq \mathrm{Lie}(G)$, through which we will identify the two, and suppose
$\mathfrak{g}_0 = \mathrm{Lie}(G)$. Since $\mathfrak{g}$ is $\Bbbk$-finite projective by (E2), (E3), 
the co-adjoint $\O(G)$-coaction on $\O(\G)^+/(\O(\G)^+)^2$ (see \eqref{eq:coad2}) is transposed to $\mathfrak{g}$,
so that $\mathfrak{g}$ is a right $G$-supermodule. The restricted right $G$-module structure on $\mathfrak{g}_1$
satisfies (F3), (F4), as was seen in the proof of Lemma \ref{lem:P}. 
To conclude $(G, \mathfrak{g}) \in (\mathrm{sHCP})_{\k}$, it remains to prove the following.

\begin{lemma}\label{lema:verifyF5}
(F5) is satisfied. 
\end{lemma} 
\begin{proof}
Let $v \mapsto \sum_i c_i\otimes v_i$ denote the left $\O(G)$-comodule structure 
$\mathfrak{g}_1 \to \O(G)\otimes \mathfrak{g}_1$ on $\mathfrak{g}_1$. 
Let $a \mapsto a^{(0)}\otimes a^{(1)}$ denote the right co-adjoint $\O(G)$-coaction
$\O(\G) \to \O(\G) \otimes \O(G)$ on $\O(\G)$. Since $\mathfrak{g}$ is
$\Bbbk$-finite projective, we have the canonical injection  
$\O(G)\otimes \mathfrak{g} = \mathrm{Hom}(\mathfrak{g}^*, \O(G))\to \mathrm{Hom}(\O(\G), \O(G))$.
Therefore, it suffices to prove
\[
\langle v^2, a^{(0)} \rangle \, a^{(1)} = \sum_i c_i^2 \, \langle v_i^2, a \rangle
+ \sum_{i<j}c_ic_j \, \langle [v_i,v_j], a \rangle
\]
for $v \in \mathfrak{g}_1$, $a \in \O(\G)$, where $\langle \ , \ \rangle$  
denotes the canonical pairing $\O(\G)^* \times \O(\G) \to \Bbbk$. This is proved as follows.
\begin{eqnarray*}
\text{LHS} &= &\langle v, (a_{(1)})^{(0)}\rangle \,
\langle v, (a_{(2)})^{(0)}\rangle \, (a_{(1)})^{(1)} (a_{(2)})^{(1)} \\
&= & \sum_{i,j} c_ic_j\, \langle v_i, a_{(1)}\rangle \, \langle v_j, a_{(2)} \rangle 
= \sum_{i,j} c_ic_j\, \langle v_iv_j, a \rangle = \text{RHS}.  
\end{eqnarray*}
\end{proof}

Let $\mathbf{P}'(\G)$ denote the thus obtained object $(G, \mathfrak{g})$ in $(\mathrm{sHCP})_{\k}$.
As in Proposition \ref{prop:P}, we see that 
$\mathbf{P}': (\mathrm{gss}\text{-}\mathrm{fsgroups})_{\k} \to (\mathrm{sHCP})_{\k}$ gives the desired
functor, since the Lie superalgebra map induced from a morphism of affine supergroups 
obviously preserves the $2$-operation.   

Let us construct a functor $\G' : (\mathrm{sHCP})_{\k}\to (\mathrm{gss}\text{-}\mathrm{fsgroups})_{\k}$.
Let $(G, \mathfrak{g}) \in (\mathrm{sHCP})_{\k}$. Then the natural right $G$-module structure on
$\mathfrak{g}_0=\mathrm{Lie}(G)$ and the given right $G$-module structure on $\mathfrak{g}_1$ amount
to a right $G$-supermodule structure on $\mathfrak{g}$, by which the super-bracket on $\mathfrak{g}$
is $G$-equivariant, as is seen as in Remark \ref{rem:compare_definitions} (2) by using (F3), (F4). 
(According to the original definition \cite[Definition 4.1.2]{G2}, the proved $G$-equivariance is
assumed as an axiom for objects in $(\mathrm{sHCP})_{\k}$. But it can be weakened to (F4), as was 
just seen.)
Using (F5), one sees as in Lemma \ref{lem:G-stable}
(indeed, more easily) that the right $G$-supermodule structure on $\mathfrak{g}$ uniquely
extends to $\mathbf{U}(\mathfrak{g})$, so that $\mathbf{U}(\mathfrak{g})$ turns into a 
Hopf-algebra object in $\mathsf{SMod}\text{-}G$. By using an isomorphism 
$U(\mathfrak{g}_0)\otimes \wedge(\mathfrak{g}_1) \simeq \U(\mathfrak{g})$ such as given 
by Corollary \ref{cora:co-split}, we can trace the argument in Section \ref{subsec:A(G,g)},
to construct a split commutative Hopf superalgebra, $\A=\A(G,\mathfrak{g})$, such that
\[
\A \simeq \mathrm{Hom}_{U(\mathfrak{g}_0)}(\U(\mathfrak{g}),\O(G)),\quad 
\overline{\A}\simeq \O(G),\quad W^{\A} \simeq \mathfrak{g}_1^*.
\] 
It follows that this $\A$ satisfies (E1)--(E3). We let $\G'(G, \mathfrak{g})$ denote the
affine supergroup corresponding to $\A$. Then one sees that 
$\G'(G, \mathfrak{g})\in (\mathrm{gss}\text{-}\mathrm{fsgroups})_{\k}$, and 
$(G, \mathfrak{g}) \mapsto \G'(G, \mathfrak{g})$ gives the desired functor. As for the fuctoriality,
note that Condition (iii) given just above Remark \ref{rema:compare_definitions}
is used to see that a morphism $(\alpha,\beta)$ in  
$(\mathrm{sHCP})_{\k}$ induces, in particular, a Hopf superalgebra map $\mathbf{U}(\mathfrak{g})
\to \mathbf{U}(\mathfrak{g}')$; see the proof of Proposition \ref{prop:functor_from_(G,g)}. 

\begin{theorem}[\text{\cite[Theorem~4.3.14]{G2}}]\label{thma:Gavarini's_equivalence}
We have a category equivalence 
\[
(\mathrm{gss}\text{-}\mathrm{fsgroups})_{\k} \approx (\mathrm{sHCP})_{\k}.
\]
In fact the functors
$\mathbf{P}'$ and $\G'$ constructed above are quasi-inverse to each other. 
\end{theorem}
\begin{proof}
To prove $\mathbf{P}'\circ \G' \simeq \mathrm{id}$, $\G' \circ \mathbf{P}'\simeq \mathrm{id}$, we can trace 
the argument of Section \ref{subsec:equivalence_theorem}
proving $\mathbf{P} \circ \G \simeq \mathrm{id}$, $\G \circ \mathbf{P} \simeq \mathrm{id}$, except in two
points. 

First, to prove $\mathbf{P}'\circ \G' \simeq \mathrm{id}$, we have to show that if 
$(G,\mathfrak{g})\in (\mathrm{sHCP})_{\k}$, and we set $\G:= \G'(G,\mathfrak{g})$, then
the natural Lie superalgebra isomorphism $\mathrm{Lie}(\G) \simeq \mathfrak{g}$ 
as given in the proof of Proposition \ref{prop:PG=id} 
preserves the $2$-operation. Note that we have a Hopf pairing  
$\U(\mathfrak{g})\times \O(\G) \to \Bbbk$ as given in \eqref{eq:UA_pairing}, and it
restricts to a non-degenerate pairing $\mathfrak{g} \times \O(\G)^+/(\O(\G)^+)^2 \to \Bbbk$,
which induces the isomorphism above. Therefore, we have the commutative 
diagram
\[
\begin{xy}
(0,0)   *++{\mathfrak{g}}  ="1",
(25,0)  *++{\mathrm{Lie}(\G)}    ="2",
(0,-15) *++{\U(\mathfrak{g})}          ="3",
(25,-15)*++{\O(\G)^*,}    ="4",
{"1" \SelectTips{cm}{} \ar @{->}^{\simeq} "2"},
{"1" \SelectTips{cm}{} \ar @{^(->} "3"},
{"2" \SelectTips{cm}{} \ar @{^(->} "4"},
{"3" \SelectTips{cm}{} \ar @{->} "4"}
\end{xy}
\]
where the arrow in the bottom is the map
induced from the Hopf pairing above. Given $v \in \mathfrak{g}_1$,
the composite $\mathfrak{g}\overset{\simeq}{\longrightarrow}\mathrm{Lie}(\G) \hookrightarrow \O(\G)^*$, 
which factors
through $\U(\mathfrak{g})$ as above, sends $v^{\langle 2 \rangle}$ to $v^2$. This proves the desired result. 
  
Second, to prove $\G' \circ \mathbf{P}'\simeq \mathrm{id}$, 
we should remark that Lemma \ref{lem:last_applied} can apply, 
since the conclusion of the lemma holds so long as
$W^{\A}$ is $\Bbbk$-finite, even if
the split commutative Hopf superalgebra $\A$ is not finitely generated.
\end{proof}

\begin{rem}\label{rema:compare_functors}
In \cite{G2}, details are not given for the following two.

(1)\ $2$-\emph{operations}.   
Condition (F5) is not explicitly given in \cite{G2}. The functor
$\Phi_g : (\mathrm{gss}\text{-}\mathrm{fsgroups})_{\k} \to (\mathrm{sHCP})_{\k}$ in \cite{G2} is
almost the same as our $\mathbf{P}'$, but it does not specify the associated
$2$-operation; see~\cite[Proposition 4.1.3]{G2}. Accordingly, it is not proved that 
$\Phi_g(G_{\mathscr{P}})\overset{\simeq}{\longrightarrow} \mathscr{P}$ preserves the $2$-operation on the 
associated Lie superalgebras; see the
first paragraph of the proof of \cite[Theorem 4.3.14]{G2}. 

(2)\ \emph{Proof of} 
$U(\mathfrak{g}_0)\otimes \wedge(\mathfrak{g}_1) \simeq \U(\mathfrak{g})$.  
This isomorphism is what was proved by our Corollary~\ref{cora:co-split}. The proof of \cite{G2}
given in the three lines above Eq.~(4.7) 
is rather sketchy, and it might overlook the localization argument used in our proof. Note that the argument
uses Proposition \ref{propa:base_extension_of_2-operation}; this last result or any equivalent one
is not given in \cite{G2}. 
\end{rem}

\section*{Acknowledgments}

The first-named author was supported by JSPS Grant-in-Aid for Scientific Research (C)~23540039. 
He thanks Fabio Gavarini who kindly answered his questions about results in \cite{G1, G2}, sending 
Errata. 
The second-named author was supported by Grant-in-Aid for JSPS Fellows 26E2022.
\vspace{8mm}

\noindent
{\bf Note.}\quad
After the present paper was accepted for publication 
the authors proved in Theorem 5.7 of the preprint
``On functor  points of affine supergroups," arXiv:~1505. 06558v2, the following:
in the situation of Section A.2 above, Conditions (E2) and (E3) implies (E1),
provided $\overline{\mathbf{A}}$ is $\Bbbk$-flat. As is remarked by Remark 5.8 (2) of the
preprint, it follows that to define the category $\mathsf{AHSA}$ in Section 4.3 above,
we may weaken Condition (C1) to (C1')~$W^{\mathbf{A}}$ is $\Bbbk$-free. For
this weakened condition (C1') together with (C2)  and (C3) turn out to ensure 
the existence of an isomorphism
$\mathbf{A} \overset{\simeq}{\longrightarrow} \overline{\mathbf{A}} \otimes \wedge(W^{\mathbf{A}})$
of left $\overline{\mathbf{A}}$-comodule superalgebras; see Definition 2.1 above.

\end{document}